\documentclass[11pt,english]{amsart}
\usepackage[T1]{fontenc}
\usepackage[latin9]{inputenc}
\usepackage{amstext}
\usepackage{amsthm}

\makeatletter
\usepackage{amsmath,amsfonts,amssymb,amsthm,epsfig}

\usepackage{color}
\usepackage[final,allcolors=blue,colorlinks=true]{hyperref}
\def\R{\mathbb R}

\def\de{\delta}
\def\ep{\epsilon}

\def\na{\nabla}
\def\Om{\Omega}  % domains
\def\De{\Delta}      % Laplacian operators
\def\La{\Lambda}

\def\cal{\mathcal}
\def\wq{\infty}
\def\pa{\partial}

\newcommand{\D}{{\rm d}}%  integral sign d

\newcommand{\medint}{-\kern -,375cm\int}         %  average integral
\newcommand{\medintinrigo}{-\kern -,315cm\int}
\newcommand{\wto}{\rightharpoonup}                %  weak convergence

\numberwithin{equation}{section}
\newtheorem{theorem}{Theorem}[section]
\newtheorem*{theorem*}{Theorem}  %Theorems without number

\newtheorem*{conclusion*}{Conclusin}

\newtheorem{corollary}[theorem]{Corollary}
\newtheorem*{corollary*}{Corollary}

\newtheorem{definition}[theorem]{Definition}

\newtheorem{lemma}[theorem]{Lemma}
\newtheorem*{lemma*}{Lemma}

\newtheorem*{notation*}{Notation}

\newtheorem{proposition}[theorem]{Proposition}
\newtheorem*{proposition*}{Proposition}
\newtheorem{remark}[theorem]{Remark}
\newtheorem*{remark*}{Remark}

\newtheorem*{example*}{Example}                %%%%% example can not use, due to conflict with Lyx.
\theoremstyle{definition}
\makeatother

\usepackage{babel}
\begin{document}
	\title[]{Quantitative regularity for minimizing intrinsic fractional harmonic maps}
	
	\author[Y.-Y. Wang, C.-L. Xiang, G.-F. Zheng]{Yu-Ying Wang, Chang-Lin Xiang, Gao-Feng Zheng$^\ast$}
	
\address[Yu-Ying Wang]{School of Mathematics and Statistics, Hubei Key Laboratory of Mathematical Sciences, Central China Normal University, Wuhan 430079,  P. R.  China}
	\email{yuyingwang@mails.ccnu.edu.cn}

	\address[Chang-Lin Xiang]{Three Gorges Mathematical Research Center, China Three Gorges University, 443002, Yichang, People's Republic of China}
	\email{changlin.xiang@ctgu.edu.cn}
	
	\address[Gao-Feng Zheng]{School of Mathematics and Statistics, Key Lab NAA-MOE, Central China Normal University, Wuhan 430079,  P. R.  China}
	\email{gfzheng@ccnu.edu.cn}
	
	\thanks{*: Corresponding author: Gao-Feng Zheng}

	\begin{abstract}
In this note, we study compactness and regularity theory of minimizing intrinsic fractional harmonic mappings introduced by Moser \cite{Moser-11} and  Roberts \cite{Roberts-18-CV}. Based on the partial regularity theory of Moser and Roberts, we first use the modified Luckhaus lemma of Roberts \cite{Roberts-18-CV}  to deduce compactness of these mappings, and then develop volume estimates of  singular sets by the quantitative stratification theory of Cheeger and Naber \cite{Cheeger-Naber-2013-CPAM}. Combining these two results lead to a global regularity estimates which, in turn, allow us to obtain an improvement of the dimension estimate of singular sets.
	\end{abstract}
	
	\maketitle
	
	{\small
		\keywords {\noindent {\bf Keywords:}  intrinsic $s$-harmonic maps; regularity theory; quantitative symmetry; singular set.}
		\smallskip
		\newline
		\subjclass{\noindent {\bf 2020 Mathematics Subject Classification:} 53C43, 35J48}
		\tableofcontents}
	\bigskip
	
	\section{Introduction and main results}
	
	\subsection{Background}  Fractional harmonic mappings have been increasingly studied in recent years, due to their deep connection with various geometrical or physical  problems, see for instance the connection with fractional Ginzburg-Landau equations given by Millot and Sire \cite{Millot-Sire-15}. In the initial works \cite{DaLio-13-CV, DaLio-Riviere-11-Adv, DaLio-Riviere-11-APDE}, Da Lio and Rivi\`{e}re studied odd order harmonic mappings on  odd dimensional Euclidean spaces, which are critical points of the functional $$L(u)=\int_{\R^k}|(-\De)^{k/4}u |^2dx$$
	 for $u\in \dot{H}^{k/2}(\R^{k}, N)$, with $k$ being an odd integer. A generalized fractional harmonic mapping on bounded domains was also introduced recently by several authors, see e.g. Millot, Pegon and Schikorra \cite{Millot-Pegon-2020-CV, Millot-Pegon-Schikorra-2021-ARMA} and the references therin, which are defined as critical points of the functional
	\[u\mapsto \iint_{\R^n\times\R^n\backslash \Om^c\times\Om^c}\frac{|u(x)-u(y)|^2}{|x-y|^{n+2s}}dxdy\]
	over some fractional Sobolev type spaces on $\Om\subset\R^n$ for $0<s<1$.
	 These researches not only generalize the classical harmonic map theory but also provide new perspectives for studying nonlocal geometric phenomena.
	
	  It was noted (see e.g. Moser \cite{Moser-11}) that, the above definitions on fractional harmonic mappings  have the disadvantage that they depend on the embedding of the target manifold in Euclidean spaces, by the famous embedding theorem of Nash \cite{Nash-1956}. From a geometrical point of view, it is natural to introduce an ``intrinsic" type fractional harmonic mapping in the sense that it does not depend on the embedding of target manifold, which will then be a bridge connecting manifold geometry with nonlocal analysis. This idea is also natural in view of the study on extrinsic and intrinsic biharmonic mappings, see Moser \cite{Moser-2008-CPDE}.
		
	Under such a consideration, Moser \cite{Moser-11} was  the first to introduce  intrinsic 1/2-harmonic mapping, and then Roberts \cite{Roberts-18-CV}  extended this concept to intrinsic $s$-harmonic mappings for $0<s<1$.
 Furthermore, they established a rich partial H\"older regularity theory in a series of papers \cite{Moser-11,Moser-Roberts-2022,Roberts-18-CV}, on both minimizing and stationary intrinsic fractional harmonic mappings. To our knowledge, no more regularity results seems to be known up to now in this respect. For instance, there has no result corresponding to the $L^p$ regularity theory for harmonic maps as that of \cite{Guo-Xiang-Zheng-2022-JMPA,Sharp-Topping-2013-TAMS}.
 To control the size of this paper, we shall not explore these problems in this paper. Instead, based on the partial regularity theory  of Moser  and  Roberts  \cite{Moser-11,Roberts-18-CV}, our aim in this work is to deduce a global $L^p$ regularity theory for minimizing intrinsic fractional harmonic maps, by which we shall also improve the dimension estimate on singular set in \cite{Roberts-18-CV}. We remark that, even for the "extrinsic" fractional harmonic mappings mentioned above,  quite few is known on global regularity theory as well, see \cite{He-X-Z-25} for some recent progress.

 Now we state precisely our setting and results. To be convenient for interested readers, we will follow the notations of Roberts \cite{Roberts-18-CV}. First fix an open subset $\mathcal{O} \subset \partial \mathbb{R}_{+}^{m+1}$ (not necessarily the entire boundary $\partial \mathbb{R}_{+}^{m+1}$), for which there exists a bounded    trace operator
\[
\left.T\right|_{\mathcal{O}}: \dot{W}_a^{1,2}\left(\mathbb{R}_{+}^{m+1}; \mathbb{R}^n\right) \rightarrow L^p\left(\mathcal{O}; \mathbb{R}^n\right)
\]
for some $p >1$; for more details see Section \ref{sec: Preliminaries}.
%See Definition \ref{def: homo-weighted Sobolev} for the space $\dot{W}_a^{1,2}\left(\mathbb{R}_{+}^{m+1}; \mathbb{R}^n\right)$. Such a trace operator can be constructed by combining Lemmata \ref{lemma:function space} and \ref{lemma:Weighted homogeneous Sobolev spaces} in Section \ref{sec: Preliminaries}, and Theorem 1 of \cite[Section 4.3]{Evans-1992}, provided that $\mathcal{O}$ is contained in the boundary of a Lipschitz domain \(\Omega \subset \mathbb{R}_{+}^{m+1}\).
Next fix a smooth closed Riemannian manifold $N$ which is isometrically embedded into $\R^n$ for some $n$. Hereafter, we always assume  the fractional order
$$s \in (0,1)$$ and denote
$$a = 1 - 2s.$$ For each map $u:\mathcal{O}\to N$ which belongs to the image of $T|_{\mathcal{O}}$,  consider the functional
\[
I^a(u,\mathcal{O}) = \inf \left\{ E^a(v) : v \in \dot{W}_a^{1,2}\big( \mathbb{R}_{+}^{m+1}; N \big),\; \left.T\right|_{\mathcal{O}} v = u \right\},
\]
where
\[
E^a(v) := \frac{1}{2} \int_{\mathbb{R}_{+}^{m+1}} z^a |\nabla v|^2  \D \mathbf{x},
\]
and
\[
\dot{W}_a^{1,2}\big( \mathbb{R}_{+}^{m+1}; N \big) := \left\{ v \in \dot{W}_a^{1,2}\big( \mathbb{R}_{+}^{m+1}; \mathbb{R}^n \big) : v(\mathbf{x}) \in N \text{ for a.e. } \mathbf{x} \in \mathbb{R}_{+}^{m+1} \right\}
\]
is a homogeneous Sobolev space equipped with the norm given by the square root of $E^a(v)$, and we write \(\mathbf{x} = (x, z)\) with \(x \in \mathbb{R}^m\), \(z \in \mathbb{R}_+\).

If we take \(\mathcal{O} = \partial \mathbb{R}_{+}^{m+1}\) and replace the target manifold \(N\) by \(\mathbb{R}\), then \(I^a\) reduces to the energy studied by Caffarelli and Silvestre \cite{Caffarelli-Silvestre-2007}. Moreover, The advantage of the functional \(I^a\) is that it is independent of the choice of isometric embedding of \(N\) into Euclidean spaces, since \(E^a\) is invariant under isometric embedding of \(N\). This is why we call the critical points of $I^a$ intrinsic fractional harmonic maps. Due to its typicality, in this work we will focus on mimimizing intrinsic fractional harmonic maps, which are defined as follows.
	
%Since \(N\) is compact, it admits a tubular neighbourhood of the form
%\[U_\delta(N) = \left\{ y \in \mathbb{R}^n : \operatorname{dist}(y, N) < \delta \right\}\]for some \(\delta = \delta(N) > 0\), and a smooth map \(\pi_N: U_\delta(N) \rightarrow N\) such that\[\left|\pi_N(y) - y\right| = \operatorname{dist}(y, N)\]for every \(y \in U_\delta(N)\), by Theorem 1 of Section 2.12.3 in \cite{Simon-1996}.

\begin{definition}\label{def:u minimising}(\cite[Definition 3.5]{Roberts-18-CV})
	We say that \(u \in \left.T\right|_{\mathcal{O}}\left( \dot{W}_a^{1,2}(\mathbb{R}_{+}^{m+1}; N) \right)\) is a \emph{(locally) minimizing intrinsic \(s\)-harmonic map}, or a \emph{local minimiser of \(I^a\)}, if for every compact set \(K \subset \mathcal{O}\) and every \(\tilde{u} \in \left.T\right|_{\mathcal{O}}\left( \dot{W}_a^{1,2}(\mathbb{R}_{+}^{m+1}; N) \right)\) such that \(u|_{\mathcal{O} \setminus K} = \tilde{u}|_{\mathcal{O} \setminus K}\), we have
	\[
	I^a(u, \mathcal{O}) \leq I^a(\tilde{u}, \mathcal{O}).
	\]
\end{definition}

%Accordingly, critical points of \(I^a\) are defined as intrinsic \(s\)-harmonic maps \cite{Moser-11,Roberts-18-CV}.

%\begin{definition}\label{def:intrinsic s-harmonic map}Let \(s \in (0,1)\) and \(u \in \mathcal{D}_a\). We say that \(u\) is an \emph{intrinsic \(s\)-harmonic map} if\[\left.\frac{\partial}{\partial t}\right|_{t=0} I^a\left(\pi_N(u + t \phi), \mathcal{O}\right) = \Lambda_a u(\phi) = 0\]for all \(\phi \in C_0^{\infty}(\mathcal{O}; \mathbb{R}^n)\), where \(a = 1 - 2s \in (-1,1)\). The spaces \(\mathcal{D}_a\) and the operator \(\Lambda_a\) are as given in Definition~\ref{def:map}.\end{definition}

For the definition of  general weak or stationary intrinsic $s$-harmonic maps, see also Moser \cite{Moser-11} and Roberts \cite{Roberts-18-CV}. It is observed (\cite{Moser-11,Roberts-18-CV}) that intrinsic \(s\)-harmonic maps arise as the boundary values of free boundary harmonic maps from \(\mathbb{R}_{+}^{m+1}\) to \(N\). Such maps may develop singularities in \(\mathcal{O}\), and full regularity cannot be expected in general. Indeed, Roberts  \cite{Roberts-18-CV} established the following fundamental regularity theory for minimizing intrinsic \(s\)-harmonic maps.

\begin{theorem}[\cite{Roberts-18-CV},Theorem 1.1]\label{thm:1.1}
Let \(\mathcal{O} \subset \partial \mathbb{R}_{+}^{m+1}\) be given as above. Suppose \(m \geq 3\) and \(s \in (0,1)\), or \(m = 2\) and \(s \in (0, 2/3)\). If \(u: \mathcal{O} \rightarrow N\) is an intrinsic \(s\)-harmonic map that locally minimizes \(I^a\) in \(\mathcal{O}\), then \(u\) is smooth outside a relatively closed set \(\Sigma \subset \mathcal{O}\) with $\cal{H}^{m-2s}(\Sigma)=0$, where $\cal{H}^{m-2s}$ denotes  the \((m - 2s)\)-dimensional Hausdorff measure.
\end{theorem}

This theorem implies that the Hausdorff dimension of \(\Sigma\) satisfies \(\dim_{\mathcal{H}} \Sigma \leq m - 2s\). However, as noted in \cite{Roberts-18-CV}, this bound may not be optimal.  In particular, results of Hardt--Lin \cite{Hardt-Lin-1989} and Duzaar--Steffen \cite{Duzaar-Steffen-1989-1, Duzaar-Steffen-1989-2} show that if a map \(v: \mathbb{R}^{m+1}_{+} \to N\) satisfies \(I^0(u) = E^0(v)\) (corresponding to $s=1/2$), then \(v\) is smooth on \(\mathcal{O}\) except on a relatively closed subset of Hausdorff dimension at most \(m - 2\), with examples in \cite[Page 5]{Duzaar-Steffen-1989-2} showing that $m-2$ is optimal. Roberts \cite[Page 3]{Roberts-18-CV} also expected that his dimension estimates in  Theorem \ref{thm:1.1} could be improved by combining his results and the arguments of Chapter 3  of Simon \cite{Simon-1996}.

%Subsequently, Moser and Roberts \cite{Moser-Roberts-2022} established partial regularity for minimizing harmonic maps with (partially) free boundary data on manifolds whose domain metric may degenerate or become singular along the free boundary. Our work adopts a framework similar to theirs. In particular, the present article investigates partial regularity up to the boundary for minimizers \(v: \mathbb{R}_{+}^{m+1} \rightarrow N\) of \(E^a\) under a free boundary condition.

%Finally, we compare the notions of intrinsic and extrinsic \(s\)-harmonic maps. Da Lio and Rivi\`{e}re \cite{DaLio-Riviere-11-APDE} first studied critical points of the functional\[L(u) = \inf \left\{ E^0(v) : v|_{\partial \mathbb{R}_{+}^2} = u,\; v \in \dot{W}^{1,2}(\mathbb{R}_{+}^2; \mathbb{R}^n),\; u(x) \in \mathbb{S}^{n-1} \text{ for a.e. } x \in \partial \mathbb{R}_{+}^2 \right\}.\]The maps considered by Da Lio and Rivi\`{e}re are \emph{extrinsic \(s\)-harmonic maps}, meaning they are critical points of an energy functional that depends on the choice of isometric embedding of \(N\) into Euclidean space. This contrasts with \emph{intrinsic \(s\)-harmonic maps}, whose energies are defined independently of any embedding. It is noteworthy that extrinsic \(s\)-harmonic maps always admit a smooth extension to \(\mathbb{R}_{+}^{m+1}\), whereas intrinsic \(s\)-harmonic maps need not.

	\subsection{Main results}

Since our aim is to derive a global regularity theory for minimizing intrinsic \(s\)-harmonic mappings, the following defined singular set
\begin{equation}\label{def: singular set}
\operatorname{sing}(u) = \{ x \in \Omega : u \text{ is not continuous in any neighborhood of } x \}
\end{equation}
plays a very important role. To study this set, we will use the quantitative differentiation approach introduced by Cheeger and Naber \cite{Cheeger-Naber-2013-CPAM}. From this, a global regularity estimate for minimizing intrinsic \(s\)-harmonic mappings will follow. This method has been successfully applied to a wide range of geometric variational problems, including harmonic mappings, biharmonic mappings, fractional harmonic mappings, varifolds, currents, harmonic map flow and mean curvature flow; see, for example, \cite{Breiner-Lamm-2015, Cheeger-H-N-2013, Cheeger-H-N-2015, Cheeger-Naber-2013-Invent, Cheeger-Naber-2013-CPAM, Naber-V-2018, Naber-V-2020-JEMS, Naber-V-V-2019} and references therein. We remark that, in the celebrated work of Naber and Valtorta \cite{Naber-V-2017}, the authors have developed a sharp quantitative stratification theory to deduce optimal global regularity theory for even stationary harmonic maps. See also  \cite{GJXZ-2024, JZ-2024} for extensions of \cite{Naber-V-2017} to biharmonic mappings.  But in this work, we will not follow this approach as there are still several fundamental theories  (such as the corresponding defect measure theory of Lin \cite{Lin-1999-Annals}) for stationary intrinsic  $s$-harmonic mappings needs to be established. Hence we will only focus on minimizing intrinsic  $s$-harmonic mappings in this work, and leave the study of stationary intrinsic  $s$-harmonic mappings in a forthcoming paper.

To state our results precisely, let us  introduce some necessary notations. For convenience, denote
\begin{equation}\label{H_int}
	H_{\text{int}}^{s}(\mathcal{O}, N) =  T|_{\mathcal{O}} \left( \dot{W}_a^{1,2}(\mathbb{R}_{+}^{m+1}; N) \right),
\end{equation}
and for any given $\Lambda > 0$, define
\begin{equation}\label{H_int_Lambda}
	H_{\Lambda}^{s}(\mathcal{O}, N) = \left\{ u \in H_{\text{int}}^{s}(\mathcal{O}, N): I^a(u, \mathcal{O}) < \Lambda \right\}.
\end{equation}
To avoid confusion with balls in \(\mathbb{R}^{m+1}\), we denote
\[
D_r(x) = \{ y \in \mathbb{R}^m : |y - x| < r \}
\]
as the ball in \(\mathbb{R}^m\) centered at \(x\) with radius \(r\), and write \(D_r = D_r(0)\). Our first result is the following:

\begin{theorem}[Integrability estimates]\label{thm: integrability}
	Let \(\Lambda > 0\). For \(m \geq 3\) let \(s \in [1/2, 1)\), and for \(m = 2\) let \(s \in [1/2, 2/3)\).  Then for every \(1 \leq p < 2\), there exists a constant \(C = C(s, m, \Lambda, p)\) such that, if \(u \in H_{\Lambda}^{s}(D_4, N)\) is a minimizing intrinsic \(s\)-harmonic map, then
	\[
	\int_{D_1} |\nabla u|^p  \D x \leq C \int_{D_1} r_u^{-p}  \D x < C.
	\]
\end{theorem}

Here, \(r_u\) denotes the regularity scale of \(u\), defined as follows. For a function \(f: \mathcal{O} \to \mathbb{R}\), the regularity scale function is defined by
\begin{equation}\label{def: regularity scale function}
	r_f(x) = \max \left\{ 0 \leq r \leq 1 : \sup_{y \in D_r(x)} r |\nabla f(y)| \leq 1 \right\},
\end{equation}
and the set of points with bad regularity scales is given by
\[
\mathcal{B}_r(f) = \{ x \in \mathcal{O} : r_f(x) < r \}.
\]
Theorem~\ref{thm: integrability} follows directly from the volume estimate below.

\begin{theorem}\label{thm: regularity scale estimate}
	Let \(\Lambda > 0\). For \(m \geq 3\) let \(s \in (0,1)\), and for \(m = 2\) let \(s \in (0, 2/3)\).  Then, for every \(\eta > 0\), there exists a constant \(C = C(s, m, \Lambda, \eta)\) such that, if \(u \in H_{\Lambda}^{s}(\mathcal{O}, N)\) is a minimizing intrinsic \(s\)-harmonic map, then
	\[
	\operatorname{Vol} \left( T_r(\mathcal{B}_r(u)) \cap D_1 \right) \leq C r^{1 - \eta}, \quad \forall  \, 0 < r < 1,
	\]
for \(s \in (0,1/2)\),
	and
	\[
	\operatorname{Vol} \left( T_r(\mathcal{B}_r(u)) \cap D_1 \right) \leq C r^{2 - \eta}, \quad \forall  \, 0 < r < 1,
	\]
for \(s \in [1/2,1)\),
	where \(T_r(A)\) denotes the \(r\)-tubular neighborhood of \(A\) in \(\mathbb{R}^m\). Consequently, the Minkowski dimension of \(\operatorname{sing}(u)\) satisfies
	\[
	\begin{aligned}
		\dim_{\mathcal{M}} \operatorname{sing}(u) &\leq m - 1  &&\text{for } s \in (0,1/2), \\
		\dim_{\mathcal{M}} \operatorname{sing}(u) &\leq m - 2  &&\text{for } s \in [1/2,1),
	\end{aligned}
	\]
	due to the inclusion \(\operatorname{sing}(u) \subset \mathcal{B}_r(u)\).
\end{theorem}

We remark that, due to the different estimates for $s\in (0,1/2)$ and $s\in [1/2,1)$, Theorem \ref{thm: integrability} only holds for $s\ge 1/2$. It is still open whether Theorem \ref{thm: integrability} holds for $0<s<1/2$.  We would also remark that, the Theorem \ref{thm:1.1} of Roberts showed that $\cal{H}^{n-2s}(\Sigma)=0$ and $u$ is continuous outside $\Sigma$. However, it is easy to know that $ \operatorname{sing}(u)\subset \Sigma$, since the set  $\Sigma$ is defined by an upper space extension function of $u$, which is  possibly bigger than the singular set of $u$.  As a result, Theorem \ref{thm: regularity scale estimate} shows that $u$ is continuous outside a set of Hausdorff dimension $m-[2s]-1$. This is a slight improvement of Theorem \ref{thm:1.1}.

%\begin{corollary}[dimension improvement]
%	Fix \(\mathcal{O} \subset \partial \mathbb{R}_{+}^{m+1}\). For \(m \geq 3\) let \(s \in (0,1)\), and for \(m = 2\) let \(s \in (0, 2/3)\). Suppose \(u: \mathcal{O} \rightarrow N\) is a locally minimizing \(s\)-harmonic map that  \(I^a\) in \(\mathcal{O}\). Then \(u\) is smooth except possibly on a relatively closed set \(\Sigma \subset \mathcal{O}\), and the Minkowski (and hence Hausdorff) dimension of \(\Sigma\) is at most \(m - 1\) for \(s \in (0,1/2)\) and at most \(m - 2\) for \(s \in [1/2,1)\).
%\end{corollary}

It would be interesting to compare our dimension estimate with that of minimizing extrinsic \(s\)-harmonic mappings obtained by Millot et al \cite{Millot-Pegon-Schikorra-2021-ARMA}, where the authors proved that the singular set of sphere-valued minimizing extrinsic \(s\)-harmonic mappings satisfying $\dim_H{\operatorname{sing}(u)}\le m-2$ for all $0<s<1$ and $m\ge 1$.
In our case, as aforementioned, the optimal upper bound for the Hausdorff dimension of the singular set is \(m-2\) for  minimizing  intrinsic 1/2-harmonic maps. Thus, the bound obtained in our results is therefore optimal for $s=1/2$. It remains open whether the bound for \(s \neq 1/2 \) is optimal or not.

To deduce the volume estimates in Theorem~\ref{thm: regularity scale estimate}, the key step is to prove the following volume estimate for the quantitative singular set \(\mathcal{S}_{\eta,r}^k(u)\) (see Definition~\ref{def: qs}).

\begin{theorem}[Volume estimate of singular set]\label{volume estimate}
	Let \(\Lambda > 0\) and \(k \in \{0,1,\dots,m-1\}\). For \(m \geq 3\) let \(s \in (0,1)\), and for \(m = 2\) let \(s \in (0, 2/3)\). Then, for every \(\eta > 0\), there exists a constant \(C = C(m, s, N, \Lambda, \eta) > 0\) such that for all minimizing intrinsic \(s\)-harmonic maps \(u \in H_{\Lambda}^{s}(\mathcal{O}, N)\) and all \(0 < r < 1\),
	\[
	\operatorname{Vol} \left( T_r(\mathcal{S}_{\eta,r}^k(u)) \cap D_1 \right) \leq C r^{m - k - \eta}.
	\]
\end{theorem}

As mentioned above, to prove this volume estimate, we employ the approach of Cheeger and Naber \cite{Cheeger-Naber-2013-CPAM}. However, unlike the situations for harmonic and biharmonic maps in \cite{Cheeger-Naber-2013-CPAM, Breiner-Lamm-2015}, we will have to use the monotonicity formula (see \eqref{monotonicity formula}) which holds on the upper half-space \(\mathbb{R}_{+}^{m+1}\). Thus we need to introduce a new type of quantitative symmetry to match the monotonicity property, and then establish quantitative cone splitting principles to construct a useful cover for the quantitative singular set \(\mathcal{S}_{\eta,r}^k(u)\). Details are provided in Section~\ref{sec: Quantitative stratification and volume estimates}.

Once Theorem~\ref{volume estimate} is established, Theorems~\ref{thm: integrability} and~\ref{thm: regularity scale estimate} follow from the following compactness theorem and an \(\varepsilon\)-regularity result (Theorem~\ref{thm: sym-to-reg}). %Details are given in Section~\ref{sec: regularity results}.

\begin{theorem}[Compactness]\label{thm: compactness}
    For $m \geq 3$ let $s \in (0,1)$, and for $m = 2$ let $s \in (0, 2/3)$. Assume that $\{u_i\}_{i \geq 1} \subset H_{\Lambda}^{s}(D_4, N)$ is a uniformly bounded sequence of minimizing intrinsic $s$-harmonic maps.
    Then there exist a subsequence $\{u_i\}$, a map $u \in H_{\text{int}}^{s}(D_4, N)$, sequences $\{v_i\}$ and $v$ in $\dot{W}_a^{1,2}(\mathbb{R}_{+}^{m+1}; N)$ with $\left.T\right|_{D_4} v_i = u_i$, $\left.T\right|_{D_4} v = u$, and $I^a(u_i, D_4) = E^a(v_i)$, such that for any open subset $\omega \subset D_4$ with $\overline{\omega} \subset D_4$ and every $B_\rho^+(\mathbf{y})$ satisfying $\overline{\partial^0 B_\rho^+(\mathbf{y})} \subset D_4$, where $\mathbf{y} = (y, 0)$, the following convergences hold:
    \[
    \begin{aligned}
        & u_i \to u \quad \text{strongly in } H_{\text{int}}^{s}(\omega, N), \\
        & v_i \to v \quad \text{strongly in } \dot{W}_a^{1,2}(B_\rho^+(\mathbf{y}); N).
    \end{aligned}
    \]
\end{theorem}

We establish this property using a modified Luckhaus lemma; see Section~\ref{sec: compactness} for details.
	
	\medskip	
	
	{\bf Notation.} Throughout the paper, we will use the following notations:
	
    $\bullet$ $a=1-2s$;

	$\bullet$ $\mathbb{R}_{+}^{m+1}=\{\mathbf{x}=(x,z):x\in \R^m, z>0\}$ denotes the $n+1$ dimensional open upper half space, the point $\mathbf{x}$ in $\mathbb{R}^{m+1}$ could also be denoted by $(x_1, x_2, \cdots, x_{m+1})$ in context for convenience;
	
	$\bullet$ $B_r(\mathbf{x})$ denotes the open ball in $\mathbb{R}^{m+1}$  with radius $r$ centered at $\mathbf{x}=(x, z)$;
	
	$\bullet$ $B_r^{+}(\mathbf{x})$ denotes the half open ball in $\mathbb{R}_{+}^{m+1}$ of radius $r$ centered at $\mathbf{x}=(x, 0)$, and simply write $B^+_r=B^+_r(0)$;
	
	$\bullet$ $D_r(x)$ denotes the the open ball/disk in $\mathbb{R}^m$ centered at $x$, and write $D_r=D_r(0)$.

	$\bullet$ For an arbitrary open set \( G \subseteq \mathbb{R}_{+}^{m+1} \), we write
\[
\partial^{+}G := \partial G \cap \mathbb{R}_{+}^{m+1}.
\]
The (relative) open set $\partial^{0}G \subseteq \partial\mathbb{R}_{+}^{m+1}$ is defined by
\[
\partial^{0}G := \left\{ \mathbf{x} \in \partial G \cap \partial \mathbb{R}_{+}^{m+1} : B_{r}^{+}(\mathbf{x}) \subseteq G \text{ for some } r > 0 \right\}.
\]
	
	$\bullet$ Finally, we identify $\mathbb{R}^m=\partial \mathbb{R}_{+}^{m+1}$; a set $A \subset \mathbb{R}^m$ is also identified with $A \times\{0\} \subset \partial \mathbb{R}_{+}^{m+1}$.

 \section{ Preliminaries}\label{sec: Preliminaries}

This section is devoted to a review of preliminary results concerning minimizing intrinsic $s$-harmonic mappings. Standard references for this material include \cite{Moser-11, Roberts-18-CV}.

\subsection{Some Sobolev spaces}\label{sec: Sobolev spaces}

We begin by introducing the space \(W_a^{1,2}(\Omega ; \mathbb{R}^n)\). As shown in \cite{Roberts-18-CV}, the Euler-Lagrange equations for \(E^a\) at critical points \(v: \mathbb{R}_{+}^{m+1} \rightarrow N\) are semilinear, with leading-order term \(\operatorname{div}(z^a \nabla v)\). To analyze solutions of such equations, we require Sobolev spaces adapted to this structure. In this section, we record the necessary function spaces and summarize some of their key analytical properties.

Let \(\Omega \subset \mathbb{R}^{m+1}\). For every \(s \in (0,1)\) and \(a = 1 - 2s \in (-1,1)\), we define the weighted Lebesgue space
\[
L_a^2(\Omega ; \mathbb{R}^n) = \left\{ f: \Omega \rightarrow \mathbb{R}^n \text{ measurable} : \int_\Omega |f|^2 |z|^a  \mathrm{d}\mathbf{x} < \infty \right\},
\]
which is a Hilbert space (see \cite{Cohn-13}, Theorem 3.4.1), where \(\mathrm{d}\mathbf{x}\) denotes Lebesgue measure on \(\mathbb{R}^{m+1}\). The inner product of \(f, g \in L_a^2(\Omega ; \mathbb{R}^n)\) is given by
\[
\langle f, g \rangle_{L_a^2(\Omega ; \mathbb{R}^n)} = \int_\Omega \langle f, g \rangle |z|^a  \mathrm{d}\mathbf{x},
\]
where \(\langle f, g \rangle\) is the Euclidean inner product in \(\mathbb{R}^n\).
Then define the weighted Sobolev space
\begin{equation}\label{def: weighted Sobolev space}
W_a^{1,2}(\Omega ; \mathbb{R}^n) = \left\{ v: \Omega \rightarrow \mathbb{R}^n : v, \partial_i v \in L_a^2(\Omega ; \mathbb{R}^n) \text{ for } i = 1, \dots, m+1 \right\},
\end{equation}
where \(\partial_i v\) denotes the weak partial derivative with respect to \(x_i\). By Proposition 2.1.2 of \cite{Turesson-2000}, this is a Hilbert space under the inner product
\[
\langle v, w \rangle_{W_a^{1,2}(\Omega ; \mathbb{R}^n)} = \int_\Omega \langle v, w \rangle |z|^a  \mathrm{d}\mathbf{x} + \int_\Omega \langle \nabla v, \nabla w \rangle |z|^a  \mathrm{d}\mathbf{x},
\]
for \(v, w \in W_a^{1,2}(\Omega ; \mathbb{R}^n)\), where \(\langle \nabla v, \nabla w \rangle = \sum_{i=1}^{m+1} \langle \partial_i v, \partial_i w \rangle\).

We further define \(W_{a,0}^{1,2}(\Omega ; \mathbb{R}^n)\) as the closure of \(C_0^\infty(\Omega ; \mathbb{R}^n)\) in \(W_a^{1,2}(\Omega ; \mathbb{R}^n)\) with respect to the norm induced by the inner product above. When \(a = 0\), we omit the subscript \(a\) in the notation.

The spaces \(W_a^{1,2}\), \(W_{a,0}^{1,2}\), and \(L_a^2\) inherit fundamental analytical properties, such as completeness, reflexivity, and separability, from their unweighted counterparts \(W^{1,2}\), \(W_0^{1,2}\), and \(L^2\), respectively.

We now state a lemma relating the spaces \(W_a^{1,2}\) and \(W^{1,p}\), which will be useful in subsequent analysis.

\begin{lemma}[\cite{Roberts-18-CV}, Lemma 2.1] \label{lemma:function space}
Let \(s \in (0,1)\) and suppose \(\Omega \subset \mathbb{R}_{+}^{m+1}\) is open and bounded.
\begin{enumerate}
    \item If \(\overline{\Omega} \subset \mathbb{R}_{+}^{m+1}\), then \(W_a^{1,2}(\Omega ; \mathbb{R}^n) = W^{1,2}(\Omega ; \mathbb{R}^n)\).
    \item If \(s \in [1/2, 1)\), then \(W_a^{1,2}(\Omega ; \mathbb{R}^n) \subset W^{1,2}(\Omega ; \mathbb{R}^n)\).
    \item If \(s \in (0, 1/2)\), then \(W_a^{1,2}(\Omega ; \mathbb{R}^n) \subset W^{1,p}(\Omega ; \mathbb{R}^n)\) for every \(1 \leq p < \frac{1}{1-s}\).
\end{enumerate}
\end{lemma}

As a consequence, we obtain the following compact embedding:
\begin{equation} \label{eq:compact embedding}
W_a^{1,2}(\Omega ; \mathbb{R}^n) \hookrightarrow \hookrightarrow L^\gamma(\Omega), \quad \text{for all } 1 < \gamma < \min\left\{ \frac{1}{1-s},\ 2 \right\}.
\end{equation}

To analyze rescaled limits of bounded sequences in these Sobolev spaces, we also require an analogue of the classical Rellich-Kondrachov compactness theorem. Away from the boundary (i.e., for \(\Omega\) such that \(\overline{\Omega} \subset \mathbb{R}_{+}^{m+1}\)), the compact embedding \(W^{1,2}\hookrightarrow \hookrightarrow L^2\) implies the compactness of \(W_a^{1,2} \hookrightarrow L_a^2\) by Lemma \ref{lemma:function space}. Near the boundary \(\partial \mathbb{R}_{+}^{m+1}\), we rely on the following lemma.

\begin{lemma}[\cite{Roberts-18-CV}, Theorem 2.5] \label{lemma:Compactness of the embedding}
Let \(r > 0\), \(\mathbf{y} \in \partial \mathbb{R}_{+}^{m+1}\), and suppose \((v_j)_{j \in \mathbb{N}}\) is a sequence in \(W_a^{1,2}(B_r^+(\mathbf{y}) ; \mathbb{R}^n)\) such that
\[
\sup_j \| v_j \|_{W_a^{1,2}(B_r^+(\mathbf{y}) ; \mathbb{R}^n)} < \infty.
\]
Then there exists a subsequence \((v_{j_k})_{k \in \mathbb{N}}\) and a function \(v \in W_a^{1,2}(B_r^+(\mathbf{y}) ; \mathbb{R}^n)\) such that:
\begin{enumerate}
    \item \(v_{j_k} \rightharpoonup v\) weakly in \(W_a^{1,2}(B_r^+(\mathbf{y}) ; \mathbb{R}^n)\),
    \item \(v_{j_k} \rightarrow v\) strongly in \(L_a^2(B_r^+(\mathbf{y}) ; \mathbb{R}^n)\),
    \item \(\displaystyle \int_{B_r^+(\mathbf{y})} z^a |\nabla v|^2  \mathrm{d}\mathbf{x} \leq \liminf_{k \to \infty} \int_{B_r^+(\mathbf{y})} z^a |\nabla v_{j_k}|^2  \mathrm{d}\mathbf{x}\).
\end{enumerate}
\end{lemma}

\begin{remark}
Lemma \ref{lemma:Compactness of the embedding} remains valid when \(B_r^+(\mathbf{y})\) is replaced by \(\mathcal{O} \times [0, r]\) for any \(r > 0\) and bounded smooth open  domain \(\mathcal{O} \subset \partial \mathbb{R}_{+}^{m+1}\).
\end{remark}

Next, we introduce weighted homogeneous Sobolev spaces. Consider the Dirichlet energy
\begin{equation}\label{eq:Dirichlet energy}
E^a(v)=\frac{1}{2} \int_{\mathbb{R}_{+}^{m+1}} z^a|\nabla v|^2 \mathrm{~d} \mathbf{x}.
\end{equation}
It is well-defined on
$$
\mathcal{D}_{+}\left(\mathbb{R}_{+}^{m+1} ; \mathbb{R}^n\right):=\left\{\phi:\phi=f|_{\mathbb{R}_{+}^{m+1}} \text { for some } f \in C_0^{\infty}\left(\mathbb{R}^{m+1} ; \mathbb{R}^n\right)\right\}.
$$
The energy $E^a$ is naturally associated to the following Sobolev space.

\begin{definition}\label{def: homo-weighted Sobolev}
Let $s\in (0,1)$. The weighted homogeneous Sobolev space $\dot{W}_a^{1,2}\left(\mathbb{R}_{+}^{m+1} ; \mathbb{R}^n\right)$ is the completion of $\mathcal{D}_{+}\left(\mathbb{R}_{+}^{m+1} ; \mathbb{R}^n\right)$ under the norm induced by the square root of $E^a$.
\end{definition}

The elements of $\dot{W}_a^{1,2}\left(\mathbb{R}_{+}^{m+1} ; \mathbb{R}^n\right)$ are, strictly speaking, equivalent classes of Cauchy sequences and it will be necessary to have tangible representatives of these classes which may take values in $N$.

\begin{lemma} (\cite[Theorem 2.3]{Roberts-18-CV}) \label{lemma:Weighted homogeneous Sobolev spaces}
 Let $m \in \mathbb{N}$ with $m \geq 2$ and $\Omega$ be an open bounded subset of $\mathbb{R}_{+}^{m+1}$. When $m \geq 3$ let $s \in (0,1)$, and when $m = 2$ let $s \in (0, 2/3)$. Then the restriction operator $I: \mathcal{D}_{+}\left(\mathbb{R}_{+}^{m+1} ; \mathbb{R}^n\right) \rightarrow W_a^{1,2}\left(\Omega ; \mathbb{R}^n\right):\left.f \mapsto f\right|_{\Omega}$ extends to a bounded linear operator
 $I: \dot{W}_a^{1,2}\left(\mathbb{R}_{+}^{m+1} ; \mathbb{R}^n\right) \rightarrow W_a^{1,2}\left(\Omega ; \mathbb{R}^n\right)$. Moreover,
\begin{equation}\label{eq:Weighted homogeneous Sobolev spaces}
\|I v\|_{W_a^{1,2}\left(\Omega ; \mathbb{R}^n\right)} \leq C\|v\|_{\dot{W}_a^{1,2}\left(\mathbb{R}_{+}^{m+1} ; \mathbb{R}^n\right)}
\end{equation}
for every $v \in \dot{W}_a^{1,2}\left(\mathbb{R}_{+}^{m+1} ; \mathbb{R}^n\right)$ where $C$ is a positive constant depending on $m$ and $\Omega$.
\end{lemma}

We conclude this subsection with a discussion of the trace operator $T|_{\mathcal{O}}$. By combining Lemmata \ref{lemma:function space} and \ref{lemma:Weighted homogeneous Sobolev spaces} with Theorem 1 of \cite[Section 4.3]{Evans-1992},  there has a bounded  trace operator $\left.T\right|_{\mathcal{O}}: \dot{W}_a^{1,2}\left(\mathbb{R}_{+}^{m+1} ; \mathbb{R}^n\right) \rightarrow$ $L^p\left(\mathcal{O} ; \mathbb{R}^n\right)$ for $p=p(a) \in(1,2]$,  whenever $\mathcal{O}$ is contained in the boundary of a Lipschitz $\Omega \subset \mathbb{R}_{+}^{m+1}$.

\subsection{Intrinsic s-harmonic maps, monotonicity and $\varepsilon$-regularity}\label{sec: intrinsic s-harmonic maps}

We assume, after translating \(N\) if necessary, that \(0 \in N\). In view of Lemma~\ref{lemma:Weighted homogeneous Sobolev spaces}, we always assume \(m \geq 2\). For \(m \geq 3\) we take \(s \in (0,1)\), and for \(m = 2\) we take \(s \in (0, 2/3)\). We define
\[
\dot{W}_a^{1,2}\left(\mathbb{R}_{+}^{m+1}; N\right) = \left\{ v \in \dot{W}_a^{1,2}\left(\mathbb{R}_{+}^{m+1}; \mathbb{R}^n\right) : v(x) \in N \text{ for almost every } x \in \mathbb{R}_{+}^{m+1} \right\},
\]
and
\[
I^a(u, \mathcal{O}) = \inf \left\{ E^a(v) : v \in \dot{W}_a^{1,2}\left(\mathbb{R}_{+}^{m+1}; N\right),\ \left.T\right|_{\mathcal{O}} v = u \right\}
\]
for \(u \in \left.T\right|_{\mathcal{O}}\left( \dot{W}_a^{1,2}\left(\mathbb{R}_{+}^{m+1}; N\right) \right)\). Recall from the introduction that \(I^a\) serves as an intrinsic energy for \(u\); it is independent of the choice of isometric embedding of \(N\) into Euclidean space. Moreover, \(I^a\) coincides with the square of the fractional Sobolev norm \(\|u\|_{\dot{H}^{s}(\partial \mathbb{R}_{+}^{m+1}; \mathbb{R}^n)}\) when \(N = \mathbb{R}^n\) and \(\mathcal{O} = \partial \mathbb{R}_{+}^{m+1}\).

For every $ u \in \left.T\right|_{\mathcal{O}}\left(\dot{W}_a^{1,2}\left(\mathbb{R}_{+}^{m+1} ; N\right)\right)$, by applying a direct method of calculus of variations, we can prove that there exists $v \in \dot{W}_a^{1,2}\left(\mathbb{R}_{+}^{m+1} ; N\right)$ with $\left.T\right|_{\mathcal{O}} v=u$, such that $I^a(u,\mathcal{O})=E^a(v)$. It is worth noting that $v$ may not be unique. Such a minimizer $v$ is referred to henceforth as a minimal harmonic map. This is due to the fact that  $v$ can be viewed as a minimal harmonic map from the Riemannian manifold $(\mathbb{R}_{+}^{m+1},g=z^\alpha \de)$ into $N$, where $\alpha=2a/(m-1)$ and $\de$ is the standard Euclidean metric: the Dirichlet energy on $(\mathbb{R}_{+}^{m+1},g)$ is precisely $E^a$.

In order to deduce regularity  for minimizing intrinsic $s$-harmonic map (see Definition \ref{def:u minimising}), we will make use of these minimal harmonic extensions. Hence we introduce the following closely related concept.

	\begin{definition}\label{def:v minimising}
 Let $v \in \dot{W}_a^{1,2}\left(\mathbb{R}_{+}^{m+1} ; N\right)$. We say that $v$ is $E^a$ minimizing, or energy minimizing, in $\mathbb{R}_{+}^{m+1}$ relative to $\mathcal{O} \subset \partial \mathbb{R}_{+}^{m+1}$, if for every compact $K \subset \mathbb{R}^{m+1}$ with $K \cap \partial \mathbb{R}_{+}^{m+1} \subset \mathcal{O}$ and for every $w \in \dot{W}_a^{1,2}\left(\mathbb{R}_{+}^{m+1} ; N\right)$ with $\left.v\right|_{\mathbb{R}_{+}^{m+1} \backslash K}=\left.w\right|_{\mathbb{R}_{+}^{m+1} \backslash K}$ we have $E^a(v) \leq E^a(w)$.
\end{definition}

%The maps considered in \cite{Schoen-Uhlenbeck-1982,Schoen-1984} belong to
%$$
%W^{1,2}(\Omega ; N)=\left\{v \in W^{1,2}\left(\Omega ; \mathbb{R}^n\right): v(x) \in N \text { for almost every } x \in B_r(y)\right\},
%$$
%for open $\Omega \subset \mathbb{R}_{+}^{m+1}$. Consider the compact Riemannian manifold $\overline{B_1(0)}$ with metric $\tilde{g}$. Recall from the introduction that the Dirichlet energy functional on $\overline{B_1(0)}$ is given by
%$$
%E_{\tilde{g}}(v)=\int_{B_1(0)}|\nabla v|_{\tilde{g}}^2 \sqrt{\operatorname{det}(\tilde{g})} \mathrm{d} x .
%$$
%
%A minimiser of $E_{\tilde{g}}$ with fixed boundary data is defined in Section 1 of \cite{Schoen-Uhlenbeck-1982} as %follows.
%
%
%\begin{definition}\label{definition:minimiser with fixed boundary data}
% Any $v \in W^{1,2}\left(B_1(0) ; N\right)$ is an $E_{\tilde{g}}$ minimising map if it satisfies $E_{\tilde{g}}(v) \leq$ $E_{\tilde{g}}(w)$ for any $w \in W^{1,2}\left(B_1(0) ; N\right)$ with $v-w \in W_0^{1,2}\left(\overline{B_1(0)} ; \mathbb{R}^n\right)$.
%\end{definition}

Roberts  \cite[Lemma 3.7]{Roberts-18-CV}  proved the following connection between local minimisers of $I^a$ and minimisers of $E^a$ relative to $\mathcal{O}$.

\begin{lemma}\label{lemma2.7}
 Suppose $u \in \left.T\right|_{\mathcal{O}}\left(\dot{W}_a^{1,2}\left(\mathbb{R}_{+}^{m+1} ; N\right)\right)$ locally minimises $I^a$ in the sense of Definition \ref{def:u minimising} and fix a minimal harmonic map $v \in \dot{W}_a^{1,2}\left(\mathbb{R}_{+}^{m+1} ; N\right)$ with $\left.T\right|_{\mathcal{O}}v=u$. Then $v$ is a minimiser of $E^a$ relative to $\mathcal{O}$.
\end{lemma}

Obviously, minimisers of $E^a$ relative to $\mathcal{O}$ are critical points of $E^a$ with respect to outer and inner variations, including those which vary their boundary data in $\mathcal{O}$. As a consequence, the following monotonicity formula holds, as proved by  Roberts \cite[Lemma 4.5]{Roberts-18-CV}, which plays a central role in the regularity theory of minimizing intrinsic $s$-harmonic maps. As that of harmonic maps, we remark that this monotonicity formula also holds for stationary intrinsic harmonic maps, which is discussed in detail in Moser and Roberts \cite{Moser-Roberts-2022} and Roberts \cite{Roberts-18-CV}.

\begin{theorem}[Monotonicity formula]\label{monotonicity formula} Suppose $v \in \dot{W}_a^{1,2}\left(\mathbb{R}_{+}^{m+1} ; N\right)$ is a minimizer of $E^a$ relative to $\mathcal{O}$. Suppose $\mathbf{y}=\left(y, 0\right) \in \mathcal{O} \times\{0\}$  and consider $B_R^{+}(\mathbf{y})$ with $\overline{\partial^0 B_R^{+}(\mathbf{y})} \subset \mathcal{O}$. Then the normalized energy function
$$
r \in\left(0, R\right] \mapsto \mathbf{\Theta}_s\left(v, \mathbf{y}, r\right):=r^{2s-m} \int_{B_r^{+}(\mathbf{y})} z^a|\nabla v|^2 \mathrm{~d} \mathbf{x}
$$
is nondecreasing. Moreover,
\begin{equation}\label{monotonicity formula}
\begin{aligned}
\boldsymbol{\Theta}_s\left(v, \mathbf{y}, r\right)-\boldsymbol{\Theta}_s\left(v, \mathbf{y}, \rho\right)=2 \int_{B_r^{+}(\mathbf{y}) \backslash B_\rho^{+}(\mathbf{y})} z^a \frac{|(\mathbf{x}-\mathbf{y}) \cdot \nabla v|^2}{|\mathbf{x}-\mathbf{y}|^{m+2-2s}} \mathrm{d} \mathbf{x}
\end{aligned}
\end{equation}
whenever $0 \leq \rho \leq r \leq R$.
\end{theorem}
Consequently, there always exists the limit
\[
\boldsymbol{\Xi}_s(v, \mathbf{y}):=\lim_{r\to 0}\boldsymbol{\Theta}_s\left(v, \mathbf{y}, r\right).
\]

Finally, let us recall the the following partial Lipschitz regularity theorem for minimal harmonic maps established by \cite{Roberts-18-CV} which is based on the monotonicity formula above.

\begin{theorem}[Partial Regularity] \label{thm: partial regularity}
 Let $s \in(0,1)$ for $m \geq 3$, and let $s \in\left(0, 2/3\right)$ for $m=2$. Let $v \in \dot{W}_a^{1,2}\left(\mathbb{R}_{+}^{m+1} ; N\right)$ be a minimiser of $E^a$ relative to $\mathcal{O}$, let $x_0 \in \partial \mathbb{R}_{+}^{m+1}$ and $B_R^{+}\left(x_0\right)$ be a half-ball with $R \leq 1$ and $\overline{\partial^0 B_R^{+}\left(x_0\right)} \subset \mathcal{O}$. There exists $\varepsilon=\varepsilon(m, N, s)>0$ such that the following holds. If
 $$R^{2s-m} \int_{B_R^{+}\left(x_0\right)} z^a|\nabla v|^2 \mathrm{~d} x \leq \varepsilon,$$
  then there are $\theta,\gamma \in(0,1)$ depending only on $m,n,s$, such that $v \in C^{0, \gamma} \left(\overline{B_{\theta R}^{+}\left(x_0\right)} ; N\right)$. Furthermore, for every $l \in \mathbb{N}$ there is a $\theta=\theta(m, N, s, l) \in(0,1)$ such that $D^{\alpha^{\prime}} v \in C^{0, 1}\left(\overline{B_{\theta R}^{+}\left(x_0\right)} ; \mathbb{R}^n\right)$ for every $\alpha^{\prime} \in \mathbb{N}_0^{m+1}$ with $\left|\alpha^{\prime}\right| \leq l$ and $\alpha_{m+1}^{\prime}=0$.
\end{theorem}

This theorem with the standard covering arguments leads to the partial regularity result in Theorem \ref{thm:1.1}. In particular, the set $\Sigma$ in Theorem \ref{thm:1.1} is characterized by
\begin{equation}\label{Sigma}
  \Sigma= \{ y\in \mathcal{O} : \boldsymbol{\Xi}_s(v, \mathbf{y})>\varepsilon, \mathbf{y}=(y,0)\}=\operatorname{sing}(v)\cap \mathcal{O}.
\end{equation}
Note that since $v$ is an upper space extension of $u$, it is quite possible that $\Sigma \supset \operatorname{sing}(u)$, where $\operatorname{sing}(u)$ is defined as in \eqref{def: singular set}.

\section{Compactness and tangent maps}\label{sec: compactness}

To apply the quantitative stratificationi theory of Cheeger-Naber \cite{Cheeger-Naber-2013-CPAM}, blow-up argument and compactness theory play a fundamental role. Hence we devote this section to the proof of the compactness Theorem \ref{thm: compactness} by following the approach of Luckhaus \cite{Luckhaus-1993} on harmonic maps, see also Simon \cite[Chapter 2]{Simon-1996}. In this respect, Roberts \cite[Section 4.5]{Roberts-18-CV} have already established a Luckhaus type lemma. First we introduce some necessary notations.

Let $\mathbb{S}^m \subset \mathbb{R}^{m+1}$ denote the $m$-dimensional unit sphere centered at the origin and equipped with the metric induced by the Euclidean metric on $\mathbb{R}^{m+1}$. Define $\mathbb{S}_{+}^m = \mathbb{S}^m \cap \mathbb{R}_{+}^{m+1}$ with the metric induced from $\mathbb{S}^m$. We let $\omega$ denote a point in $\mathbb{S}^m \subset \mathbb{R}^{m+1}$ or $\mathbb{S}_{+}^m \subset \mathbb{R}_{+}^{m+1}$ and write $\mathrm{d} \omega$ for the volume element corresponding to the induced metric.

\begin{definition}
    Let $\varepsilon > 0$ and $\rho > 0$. Suppose $S = \rho \mathbb{S}^m$ and $V_{\varepsilon} = B_{\rho+\varepsilon}(0) \setminus B_{\rho-\varepsilon}(0)$, or $S = \rho \mathbb{S}_{+}^m$ and $V_{\varepsilon} = B_{\rho+\varepsilon}^{+}(0) \setminus B_{\rho-\varepsilon}^{+}(0)$. An element $v \in L_a^2(S; \mathbb{R}^n)$ is said to be in $W_a^{1,2}(S; \mathbb{R}^n)$ if the map $v(\rho \frac{x}{|x|}) \in W_a^{1,2}(V_{\varepsilon}; \mathbb{R}^n)$ for some $\varepsilon > 0$. An element $v \in L_a^2(S \times [a, b]; \mathbb{R}^n)$, with $a < b$ real numbers, is said to be in $W_a^{1,2}(S \times [a, b]; \mathbb{R}^n)$ if the map $v(\rho \frac{x}{|x|}, s) \in W_a^{1,2}(V_{\varepsilon} \times [a, b]; \mathbb{R}^n)$ for some $\varepsilon > 0$. If $N \subset \mathbb{R}^n$ is compact, we say $v$ is in $W_a^{1,2}(S; N)$ or $W_a^{1,2}(S \times [a, b]; N)$ if $v$ is in $W_a^{1,2}(S; \mathbb{R}^n)$ or $W_a^{1,2}(S \times [a, b]; \mathbb{R}^n)$, respectively, and $v(x) \in N$ for almost every $x \in S$.
\end{definition}

We now state the modified Luckhaus lemma of Roberts \cite[Lemma 4.16]{Roberts-18-CV}.

\begin{lemma}\label{lemma:modified Luckhaus}
    Let $m \geq 2$ and $s \in (0,1)$. Let $N$ be a compact subset of $\mathbb{R}^n$, and suppose $u, v \in W_a^{1,2}(\mathbb{S}_{+}^m; N)$. Then for all $\varepsilon \in (0,1)$, there is a $w \in W_a^{1,2}(\mathbb{S}_{+}^m \times [0, \varepsilon]; \mathbb{R}^n)$ such that $w$ agrees with $u$ on $\mathbb{S}_{+}^m \times \{0\}$ and $v$ on $\mathbb{S}_{+}^m \times \{\varepsilon\}$ in the sense of traces and satisfies the following. Let $\bar{D}$ be the gradient on $\mathbb{S}_{+}^m \times [0, \varepsilon]$ and $D$ the gradient on $\mathbb{S}_{+}^m$. Then $w = w(\omega, l)$ satisfies
    \begin{equation}\label{eq:Luckhaus1}
        \begin{aligned}
            & \int_{\mathbb{S}_{+}^m \times [0, \varepsilon]} \omega_{m+1}^a |\bar{D} w|^2  \mathrm{d}\omega  \mathrm{d}l \\
            & \quad \leq C_1 \varepsilon \int_{\mathbb{S}_{+}^m} \omega_{m+1}^a (|D u|^2 + |D v|^2)  \mathrm{d}\omega + \frac{C_1}{\varepsilon} \int_{\mathbb{S}_{+}^m} \omega_{m+1}^a |u - v|^2  \mathrm{d}\omega,
        \end{aligned}
    \end{equation}
    where $C_1 = C_1(m, s)$. Furthermore, $w$ satisfies
    \begin{equation}\label{eq:Luckhaus2}
        \begin{aligned}
            & \operatorname{dist}^2(w(\omega, l), N) \\
            & \leq \frac{C_2}{\varepsilon^{m + \frac{a}{2} + \frac{|a|}{2}}} \left( \int_{\mathbb{S}_{+}^m} \omega_{m+1}^a (|D u|^2 + |D v|^2)  \mathrm{d}\omega \right)^{\frac{1}{q}} \left( \int_{\mathbb{S}_{+}^m} \omega_{m+1}^a |u - v|^2  \mathrm{d}\omega \right)^{1 - \frac{1}{q}} \\
            & \quad + \frac{C_2}{\varepsilon^{m + 1 + \frac{a}{2} + \frac{|a|}{2}}} \int_{\mathbb{S}_{+}^m} \omega_{m+1}^a |u - v|^2  \mathrm{d}\omega,
        \end{aligned}
    \end{equation}
    for almost every $(\omega, l) \in \mathbb{S}_{+}^m \times [0, \varepsilon]$, where $C_2 = C_2(m, s)$ and $q$ satisfies the following: If $s \in [1/2, 1)$, then \eqref{eq:Luckhaus2} holds for $q = 2$; if $s \in (0, 1/2)$, then for any $p \in (1, \frac{1}{1-s})$, there exists $q \in \{2, p\}$ such that \eqref{eq:Luckhaus2} holds.
\end{lemma}

Now we establish a useful corollary of the modified Luckhaus lemma. First, we mention the following important fact about slicing by the radial distance function.

\begin{lemma}
    Suppose $z^a g \geq 0$ is integrable on $B^+_\rho(\mathbf{y})$. By virtue of the general identity
    \[
    \int_{B^+_\rho(\mathbf{y}) \setminus B^+_{\rho/2}(\mathbf{y})} z^a g = \int_{\rho/2}^\rho \left( \int_{\partial^+ B_\sigma(\mathbf{y})} z^a g \right)  \mathrm{d}\sigma,
    \]
    we see that for each $\theta \in (0,1)$,
    \begin{equation}\label{Luckhaus corollary1}
        \int_{\partial^+ B_\sigma(\mathbf{y})} z^a g \leq C \theta^{-1} \rho^{-1} \int_{B^+_\rho(\mathbf{y}) \setminus B^+_{\rho/2}(\mathbf{y})} z^a g
    \end{equation}
    for all $\sigma \in (\frac{\rho}{2}, \rho)$ with the exception of a set of measure $\frac{\theta \rho}{2}$. (Indeed, otherwise the reverse inequality would hold on a set of measure $> \frac{\theta \rho}{2}$, and by integration this would give
    \[
    \int_{B^+_\rho(\mathbf{y}) \setminus B^+_{\rho/2}(\mathbf{y})} z^a g < \int_{\rho/2}^\rho \left( \int_{\partial^+ B_\sigma(\mathbf{y})} z^a g \right)  \mathrm{d}\sigma = \int_{B^+_\rho(\mathbf{y}) \setminus B^+_{\rho/2}(\mathbf{y})} z^a g,
    \]
    a contradiction.)

    Furthermore, if $w \in W_a^{1,2}(\Omega, \mathbb{R})$, then for each ball $\overline{B}^+_\rho(\mathbf{y}) \subset \Omega \subset \mathbb{R}^{m+1}$ with $\mathbf{y} = (y, 0)$ and each $\theta \in (0, 1)$,
    \begin{equation}\label{Luckhaus corollary2}
        \begin{aligned}
            & w_{(\sigma)} \in W_a^{1,2}(\mathbb{S}_{+}^m; \mathbb{R}) \text{ and } \int_{\mathbb{S}_{+}^m} \omega_{m+1}^a |D_\omega w_{(\sigma)}|^2  \mathrm{d}\omega \\
            & \quad \leq \sigma^{1 + 2s - m} \int_{\partial^+ B_\sigma(\mathbf{y})} z^a |\nabla w|^2 \leq 2 \theta^{-1} \rho^{2s - m} \int_{B^+_\rho(\mathbf{y}) \setminus B^+_{\rho/2}(\mathbf{y})} z^a |\nabla w|^2
        \end{aligned}
    \end{equation}
    for all $\sigma \in (\rho/2, \rho)$, with the exception of a set of $\sigma$ of measure $\frac{\theta \rho}{2}$, where $w_{(\sigma)}$ is defined by $w_{(\sigma)}(\omega) \equiv w(\mathbf{y} + \sigma \omega)$, $\omega \in \mathbb{S}_{+}^m$, and where $D_\omega$ means gradient on $\mathbb{S}_{+}^m$.
\end{lemma}

\begin{corollary}\label{corollary:Luckhaus corollary}
    Suppose $N$ is a compact subset of $\mathbb{R}^n$ and $\Lambda > 0$. There exist $\delta_0 = \delta_0(m, N, s, \Lambda)$ and $C = C(m, N, s, \Lambda)$ such that the following hold: If $\varepsilon \in (0, \delta_0]$, and if $u, v \in W_a^{1,2}(B^+_{(1+\varepsilon)\rho}(\mathbf{y}) \setminus B^+_\rho(\mathbf{y}); N)$ satisfy
    \[
    \rho^{2s - m} \int_{B^+_{(1+\varepsilon)\rho}(\mathbf{y}) \setminus B^+_\rho(\mathbf{y})} z^a (|\nabla u|^2 + |\nabla v|^2) \leq \Lambda
    \]
    and
    \[
    \varepsilon^{-(m + 1 + \frac{a}{2} + \frac{|a|}{2}) \cdot \frac{q}{q-1}} \rho^{-(m + 2 - 2s)} \int_{B^+_{(1+\varepsilon)\rho}(\mathbf{y}) \setminus B^+_\rho(\mathbf{y})} z^a |u - v|^2 < \delta_0^2,
    \]
    where $\mathbf{y} = (y, 0)$ and $q$ is as in Theorem \ref{lemma:modified Luckhaus}, then there is $w \in W_a^{1,2}(B^+_{(1+\varepsilon)\rho}(\mathbf{y}) \setminus B^+_\rho(\mathbf{y}); N)$ such that $w$ agrees with $u$ on $\partial^+ B_\rho(\mathbf{y})$ and $w$ agrees with $v$ on $\partial^+ B_{(1+\varepsilon)\rho}(\mathbf{y})$ in the sense of traces and satisfies:
    \[
    \begin{aligned}
        & \rho^{2s - m} \int_{B^+_{(1+\varepsilon)\rho}(\mathbf{y}) \setminus B^+_\rho(\mathbf{y})} z^a |\nabla w|^2 \\
        & \quad \leq C \rho^{2s - m} \int_{B^+_{(1+\varepsilon)\rho}(\mathbf{y}) \setminus B^+_\rho(\mathbf{y})} z^a (|\nabla u|^2 + |\nabla v|^2) \\
        & \quad \quad + C \varepsilon^{-2} \rho^{-(m + 2 - 2s)} \int_{B^+_{(1+\varepsilon)\rho}(\mathbf{y}) \setminus B^+_\rho(\mathbf{y})} z^a |u - v|^2.
    \end{aligned}
    \]
\end{corollary}

\begin{proof}
    We first note that by \eqref{Luckhaus corollary1} and \eqref{Luckhaus corollary2}, there is a set of $\sigma \in (\rho, (1 + \frac{\varepsilon}{2})\rho)$ of positive measure such that
    \begin{equation}\label{Luckhaus corollary3}
        \sigma^{1 + 2s - m} \int_{\partial^+ B_\sigma(\mathbf{y})} z^a (|\nabla u|^2 + |\nabla v|^2) \leq C \rho^{2s - m} \varepsilon^{-1} \int_{B^+_{(1+\varepsilon)\rho}(\mathbf{y}) \setminus B^+_\rho(\mathbf{y})} z^a (|\nabla u|^2 + |\nabla v|^2)
    \end{equation}
    and
    \begin{equation}\label{Luckhaus corollary4}
        \begin{aligned}
            & \sigma^{-(m + 1) + 2s} \int_{\partial^+ B_\sigma(\mathbf{y})} z^a |u - v|^2 \\
            & \quad \leq C \rho^{-(m + 2) + 2s} \varepsilon^{-1} \int_{B^+_{(1+\varepsilon)\rho}(\mathbf{y}) \setminus B^+_\rho(\mathbf{y})} z^a |u - v|^2 \\
            & \quad \leq C \delta_0^2 \varepsilon^{(m + 1 + \frac{a}{2} + \frac{|a|}{2}) \cdot \frac{q}{q-1} - 1}.
        \end{aligned}
    \end{equation}
    By \eqref{Luckhaus corollary2}, we know that almost all of these $\sigma$ can be selected so that $u, v \in W_a^{1,2}(\partial^+ B_\sigma(\mathbf{y}); N)$. Now we can apply the modified Luckhaus lemma \ref{lemma:modified Luckhaus} (with $\varepsilon/4$ in place of $\varepsilon$) to the functions $\tilde{u}(\omega) \equiv u(\sigma \omega)$ and $\tilde{v}(\omega) = v(\sigma \omega)$, thus giving $\widetilde{w}$ on $S^{n-1} \times [0, \varepsilon/4]$ such that $\tilde{w}$ agrees with $\tilde{u}$ on $\mathbb{S}_{+}^m \times \{0\}$ and $\tilde{v}$ on $\mathbb{S}_{+}^m \times \{\varepsilon/4\}$ in the sense of traces,
    \begin{equation}\label{Luckhaus corollary5}
        \begin{aligned}
            & \int_{\mathbb{S}_{+}^m \times [0, \varepsilon/4]} \omega_{m+1}^a |\bar{D} \tilde{w}|^2 \\
            & \quad \leq C \varepsilon \int_{\mathbb{S}_{+}^m} \omega_{m+1}^a (|D \tilde{u}|^2 + |D \tilde{v}|^2) + C \varepsilon^{-1} \int_{\mathbb{S}_{+}^m} \omega_{m+1}^a |\tilde{u} - \tilde{v}|^2 \\
            & \quad \leq C \varepsilon \sigma^{1 + 2s - m} \int_{\partial^+ B_\sigma} z^a (|\nabla u|^2 + |\nabla v|^2) + C \varepsilon^{-1} \sigma^{-(m + 1) + 2s} \int_{\partial^+ B_\sigma} z^a |u - v|^2 \\
            & \quad \leq C \rho^{2s - m} \int_{B^+_{(1+\varepsilon)\rho}(\mathbf{y}) \setminus B^+_\rho(\mathbf{y})} z^a (|\nabla u|^2 + |\nabla v|^2) \\
            & \quad \quad + C \varepsilon^{-2} \rho^{-(m + 1) + 2s - 1} \int_{B^+_{(1+\varepsilon)\rho}(\mathbf{y}) \setminus B^+_\rho(\mathbf{y})} z^a |u - v|^2,
        \end{aligned}
    \end{equation}
    by \eqref{Luckhaus corollary3} and \eqref{Luckhaus corollary4}, and
    \begin{equation}\label{Luckhaus corollary6}
        \begin{aligned}
            & \sup \operatorname{dist}^2(\tilde{w}, N) \\
            & \quad \leq C \left( \int_{\mathbb{S}_{+}^m} \omega_{m+1}^a (|D \tilde{u}|^2 + |D \tilde{v}|^2) \right)^{1/q} \left( \varepsilon^{-(m + \frac{a}{2} + \frac{|a|}{2}) \cdot \frac{q}{q-1}} \int_{\mathbb{S}_{+}^m} \omega_{m+1}^a |\tilde{u} - \tilde{v}|^2 \right)^{1 - 1/q} \\
            & \quad \quad + C \varepsilon^{-(m + 1 + \frac{a}{2} + \frac{|a|}{2})} \int_{\mathbb{S}_{+}^m} \omega_{m+1}^a |\tilde{u} - \tilde{v}|^2.
        \end{aligned}
    \end{equation}
    Again by \eqref{Luckhaus corollary3} and \eqref{Luckhaus corollary4}, the right side here is $\leq C \delta_0^{2(1 - 1/q)}$ with $C = C(m, s, N, \Lambda)$. Now we can suppose (by taking a smaller $\delta_0 = \delta_0(m, s, N, \Lambda)$ if necessary) that $C \delta_0 \leq \alpha$, where $\alpha > 0$ is such that the nearest point projection $\Pi$ onto $N$ is well-defined and smooth in $N_\alpha \equiv \{x \in \mathbb{R}^n : \operatorname{dist}(x, N) \leq \alpha\}$. Now we can define a suitable function $w$, first on the ball $B_{(1+\varepsilon/2)\sigma}(y)$; we let $w(x) = \Pi \circ \tilde{w}(\omega, r/\sigma - 1)$, with $r = |x - y| \in (\sigma, (1 + \varepsilon/4)\sigma)$, $w(x) \equiv u(x)$ for $|x - y| \leq \sigma$, $w(r \omega) = v(\psi(r) \omega)$ for $r \in ((1 + \varepsilon/4)\sigma, (1 + \varepsilon/2)\sigma)$, where $\psi(t)$ is a $C^1(\mathbb{R})$ function with the properties $\psi((1 + \varepsilon/4)\sigma) = \sigma$, $\psi((1 + \varepsilon/2)\sigma) = (1 + \varepsilon/2)\sigma$ and $t |\psi'(t)| \leq 2$ for $t \in ((1 + \varepsilon/4)\sigma, (1 + \varepsilon/2)\sigma)$. In view of \eqref{Luckhaus corollary5}, it is straightforward to check that this satisfies the inequality stated in the lemma.
\end{proof}

Now, we prove the Theorem \ref{thm: compactness}.

	\begin{proof}[Proof of Theorem \ref{thm: compactness}]
We may assume that there exist a subsequence $\left\{v_{j^{\prime}}\right\}$ (henceforth denoted simply $\left\{v_j\right\}$) and a function $v \in \dot{W}_{a }^{1,2}(\mathbb{R}_{+}^{m+1} ; N)$ satisfying $\left.T\right|_{D_4} v_i=u_i$, $\left.T\right|_{D_4} v=u$ and $I^a(u_i,D_4)=E^a(v_i)<\infty$ such that $v_i\wto v$ in $\dot{W}_a^{1,2}\left(\mathbb{R}_{+}^{m+1} ; N\right)$, $v_i\wto v $ in $W_a^{1,2}\left(B_{4}^+ ; N\right)$, $v_i\to v $ in $L_a^2\left(B_{4}^+ ; N\right)$, $u_i\wto u$ in $H_{int}^{s}(D_4,N)$. %where $\Omega$ is an open bounded subset of $\mathbb{R}_{+}^{m+1}$ with $\overline{\partial^0\Omega} \subset  D_4$.
Let $\overline{\partial^0 B^{+}_{\rho_{0}}\left(\mathbf{y}\right)} \subset D_4$, $\mathbf{y}=(y, 0)$ and let $\delta>0$ and $\theta \in(0,1)$ be given. Choose any $M \in$ $\{1.2, \ldots\}$ such that
 $$\limsup _{j \rightarrow \infty}~ \rho_0^{2s-m} \int_{B^+_{\rho_0}(\mathbf{y})}z^a|\nabla v_j|^2\leq\limsup _{j \rightarrow \infty}~ \rho_0^{2s-m} \int_{\mathbb{R}_{+}^{m+1}}z^a|\nabla v_j|^2<M \delta.$$
  We note that if $\varepsilon \in(0,(1-\theta) / M)$ we must have some integer $\ell \in\{2, \ldots . M\}$ such that
$$
\rho_0^{2s-m} \int_{B^+_{\rho_0(\theta+\ell \varepsilon)}(\mathbf{y}) \backslash B^+_{\rho_0(\theta+(\ell-2) \varepsilon)}(\mathbf{y})}z^a\left|\nabla v_j\right|^2<\delta
$$
for infinitely many $j$, because otherwise we get that $\rho_0^{2s-m} \int_{B^+_{\rho_0(\mathbf{y})}}z^a\left|\nabla v_j\right|^2>M \delta$ for all sufficiently large $j$ by summation over $\ell$, contrary to the definition of $M$. Thus choosing such an $\ell$, letting $\rho=\rho_0(\theta+(\ell-2) \varepsilon)$, and noting that $\rho(1+\varepsilon) \leq \rho_0(\theta+\ell \varepsilon)<$ $\rho_0$, we get $\rho \in\left(\theta \rho_0, \rho_0\right)$ such that
\begin{equation}\label{eq:compact1}
\rho_0^{2s-m} \int_{B^+_{\rho(1+\varepsilon)(\mathbf{y})} \backslash B^+_\rho(\mathbf{y})}z^a\left|\nabla v_{j^{\prime}}\right|^2<\delta
\end{equation}
for some subsequence $j^{\prime}$. Of course then by weak convergence of $\nabla v_{j^{\prime}}$ to $\nabla v$ we also have
\begin{equation}\label{eq:compact2}
\rho_0^{2s-m} \int_{B^+_{\rho(1+\varepsilon)}(\mathbf{y}) \backslash B^+_\rho(\mathbf{y})}z^a|\nabla v|^2 \leq \delta.
\end{equation}
Now, by Corollary \ref{corollary:Luckhaus corollary}, since $\int_{B^+_{\rho_0}(\mathbf{y})}z^a\left|v-v_{j^{\prime}}\right|^2 \rightarrow 0$, for sufficiently large $j^{\prime}$ we can find $w_{j^{\prime}} \in W_{a}^{1.2}\left(B^+_{\rho(1+\varepsilon)}(\mathbf{y}) \backslash B^+_\rho(\mathbf{y}) ; N\right)$ such that $w_{j^{\prime}}$ agrees with $v$ on $\partial^+ B_\rho(\mathbf{y})$ and $w_{j^{\prime}}$ agrees with $v_{j^{\prime}}$ on $\partial^+ B_{\rho(1+\varepsilon)}(\mathbf{y})$ in the sense of traces, and
\begin{equation}\label{eq:compact3}
\begin{aligned}
& \rho^{2s-m} \int_{B^+_{(1+\varepsilon) \rho}(\mathbf{y}) \backslash B^+_\rho(\mathbf{y})}z^{a}\left|\nabla w_{j^{\prime}}\right|^2  \\
& \quad \leq C \rho^{2s-m} \int_{B^+_{(1+\varepsilon) \rho}(\mathbf{y}) \backslash B^+_\rho(\mathbf{y})}z^{a}\left(|\nabla v|^2+\left|\nabla v_{j^{\prime}}\right|^2+\varepsilon^{-2} \rho^{-2}\left|v-v_{j^{\prime}}\right|^2\right),
\end{aligned}
\end{equation}
where $C$ depends only on $m, s, N$. Now let $h \in W_{a}^{1.2}\left(B_{\theta \rho_0}^+(\mathbf{y}) ; N\right)$ with $h$ agrees with $v$ on  $\partial ^+B_{\theta \rho_0}(\mathbf{y})$ in the sense of traces, extend $h$ to give $\tilde{h}$ in $W_{a}^{1.2}\left(\mathbb{R}_{+}^{m+1} ; N\right)$ by taking $\tilde{h} \equiv v$ on $\mathbb{R}_{+}^{m+1} \backslash B_{\theta \rho_0}^+(\mathbf{y})$, and let $\tilde{v}_{j^{\prime}}$ be defined by

$$
\tilde{v}_{j^{\prime}}= \begin{cases}v_{j^{\prime}} & \text { on } \mathbb{R}_{+}^{m+1}  \backslash B^+_{(1+\varepsilon) \rho}(\mathbf{y}) \\ w_{j^{\prime}} & \text { on } B^+_{(1+\varepsilon) \rho}(\mathbf{y}) \backslash B^+_\rho(\mathbf{y}) \\ \tilde{h} & \text { on } B^+_\rho(\mathbf{y}) .\end{cases}
$$
Then by Lemma \ref{lemma2.7}, the minimizing property of $v_{j^{\prime}}$ implies
\begin{equation}\label{eq:compact4}
\begin{aligned}
\int_{B^+_{(1+\varepsilon) \rho}(\mathbf{y})}z^{a}\left|\nabla v_{j^{\prime}}\right|^2 & \leq \int_{B^+_{(1+\varepsilon) \rho}(\mathbf{y})}z^{a}\left|\nabla \widetilde{v}_{j^{\prime}}\right|^2 \\
& \leq \int_{B^+_\rho(\mathbf{y})}z^{a}|\nabla \widetilde{h}|^2+\int_{B^+_{(1+\varepsilon) \rho}(\mathbf{y}) \backslash B^+_\rho(\mathbf{y})}z^{a}\left|\nabla w_{j^{\prime}}\right|^2,
\end{aligned}
\end{equation}
and hence by \eqref{eq:compact1}, \eqref{eq:compact2}, \eqref{eq:compact3} and \eqref{eq:compact4}
\begin{equation}\label{eq:compact5}
\begin{aligned}
\rho^{2s-m} \int_{B^+_\rho(\mathbf{y})}z^{a}|\nabla v|^2
&\leq \liminf _{j \rightarrow \infty} \rho^{2s-m} \int_{B^+_\rho(\mathbf{y})}z^{a}\left|\nabla v_{j^{\prime}}\right|^2\\
&\leq \rho^{2s-m} \int_{B^+_\rho(\mathbf{y})}z^{a}|\nabla \widetilde{h}|^2+C \delta,
\end{aligned}
\end{equation}
 where $C=C(n, N, s)$. Since $\delta>0$ was arbitrary, we obtain
%\begin{equation}\label{eq:v minimising1}
$$
\rho^{2s-m} \int_{B^+_{\theta\rho_0(\mathbf{y})}}z^{a}|\nabla v|^2 \leq \rho^{2s-m} \int_{B^+_{\theta\rho_0(\mathbf{y})}}z^{a}|\nabla h|^2.
$$
%\end{equation}
Then the arbitrariness of $\theta$ and $\rho_0$ implies that for all balls $B^+_{\rho}(\mathbf{y})$ with $\overline{\partial^0 B^{+}_{\rho}\left(\mathbf{y}\right)} \subset D_4$, we have
\begin{equation}\label{eq:v minimising}
 \int_{B^+_{\rho(\mathbf{y})}}z^{a}|\nabla v|^2 \leq \int_{B^+_{\rho(\mathbf{y})}}z^{a}|\nabla h|^2.
\end{equation}

 To prove that the convergence is strong we note that if we use \eqref{eq:compact5} with $h=v$, then we can conclude

$$
\liminf _{j \rightarrow \infty} \rho^{2s-m} \int_{B^+_\rho(\mathbf{y})}z^{a}\left|\nabla v_{j^{\prime}}\right|^2 \leq \rho^{2s-m} \int_{B^+_\rho(\mathbf{y})}z^{a}|\nabla v|^2+C \delta,
$$
and hence, in view of the arbitrariness of $\theta$ and $\delta$,
$$
\rho^{2s-m} \liminf _{j \rightarrow \infty} \int_{B^+_{\rho_1}(\mathbf{y})}z^{a}\left|\nabla v_j\right|^2 \leq \rho^{2s-m} \int_{B^+_{\rho_0}(\mathbf{y})}z^{a}|\nabla v|^2,
$$
for each $\rho_1<\rho_0$. Evidently it follows from this (keeping in mind the arbitrariness of $\rho_0$ ) that
$$
\liminf _{j \rightarrow \infty} \int_{B^+_\rho(\mathbf{y})}z^{a}|\nabla v_j|^2 \leq \int_{B^+_\rho(\mathbf{y})}z^{a}|\nabla v|^2
$$
for every ball $B^+_\rho(\mathbf{y})$ such that $\overline{\partial^0 B_\rho^{+}\left(\mathbf{y}\right)} \subset D_4$. Then since
$$
\int_{B^+_\rho(\mathbf{y})}z^{a}|\nabla v_j-\nabla v|^2 \equiv \int_{B^+_\rho(\mathbf{y})}z^{a}|\nabla v|^2+\int_{B^+_\rho(\mathbf{y})}z^{a}\left|\nabla v_j\right|^2-2 \int_{B^+_\rho(\mathbf{y})}z^{a} \nabla v \cdot \nabla v_j,
$$
we can evidently select a subsequence which converges strongly to $\nabla v$ on $B^+_\rho(\mathbf{y})$. It implies $u_{j^{\prime}}$ converges strongly to $u$ in $H_{int}^{s}(\overline{\partial^0 B_\rho^{+}\left(\mathbf{y}\right)},N)$. Since this holds for arbitrary $\overline{\partial^0 B_\rho^{+}\left(\mathbf{y}\right)} \subset D_4$. it is then easy to see (by covering $D_4$ by a countable collection of balls $\partial^0 B^+_{\rho_j}\left(\mathbf{y}_{j}\right)$ such that $\overline{\partial^0 B^+_{\rho_j}\left(\mathbf{y}_{j}\right)} \subset D_4$) that there is a subsequence such that $u_{j^{\prime}}$ converges strongly locally in $H_{int}^{s}(D_4,N)$.
\end{proof}

Applying Theorem~\ref{thm: compactness} to rescaled mappings, we obtain the following corollary.

\begin{corollary}\label{rmk:tangent map}
    For $m \geq 3$, let $s \in (0,1)$, and for $m = 2$, let $s \in (0, 2/3)$. If $u \in H_{\text{int}}^{s}(\mathcal{O}, N)$ is a minimizing intrinsic $s$-harmonic map, then for every $x \in \mathcal{O}$ and every sequence $r_k \to 0$, there exist a subsequence $r_k' \to 0$, a map $\varphi: \mathbb{R}^m \to N$ that is $0$-homogeneous at the origin, and a map $\Phi: \mathbb{R}_+^{m+1} \to N$ whose restriction on $\pa\R^{m+1}_+$ is $0$-homogeneous at the origin, such that $\left.T\right|_{D_r} \Phi = \varphi$ and
    \[
    \begin{aligned}
        & u_{x, r_k'} \to \varphi \quad \text{strongly in } H_{\text{int}}^{s}(D_r, N), \\
        & v_{\mathbf{x}, r_k'} \to \Phi \quad \text{strongly in } \dot{W}_a^{1,2}(B_r^+; N),
    \end{aligned}
    \]
    for all $r > 0$, where $\mathbf{x} = (x, 0)$, $u_{x, r}(y) = u(x + ry)$.
\end{corollary}

Now we can introduce the concept of tangent map.

\begin{definition}
     Let $s \in (0,1)$ for $m \geq 3$, and  let $s \in (0, 2/3)$ for $m = 2$. Let $u \in H_{\text{int}}^{s}(\mathcal{O}, N)$ be a minimizing intrinsic $s$-harmonic map. The map $\varphi$ derived above is called a \emph{tangent map} of $u$ at $x \in \mathcal{O}$.
\end{definition}

To conclude this section, we describe the classical stratification of the singular set based on the symmetry properties of tangent maps.

\begin{definition}[Symmetry]\label{def: classical symmetry}
    Given a measurable map $\varphi: \mathbb{R}^m \to \mathbb{R}$, we say that:
    \begin{enumerate}
        \item $\varphi$ is \emph{$0$-homogeneous} or \emph{$0$-symmetric} with respect to a point $p \in \mathbb{R}^m$ if
        \[
        \varphi(p + \lambda v) = \varphi(p + v) \quad \text{for all } \lambda > 0 \text{ and } v \in \mathbb{R}^m.
        \]
        \item $\varphi$ is \emph{$k$-symmetric} if it is $0$-symmetric with respect to the origin and translation-invariant with respect to a $k$-dimensional subspace $V \subset \mathbb{R}^m$, i.e.,
        \[
        \varphi(x + v) = \varphi(x) \quad \text{for all } x \in \mathbb{R}^m,\ v \in V.
        \]
    \end{enumerate}
\end{definition}

In view of the symmetric property of $\Phi: \mathbb{R}_+^{m+1} \to N$ in Corollary \ref{rmk:tangent map}, we also introduce the following definition for convenience.

\begin{definition}[Relative/Boundary symmetry]\label{def: Relative/boundary symmetry}
	A map $h: \overline{\mathbb{R}_{+}^{m+1}} \to \mathbb{R}$ is called \emph{boundary $k$-symmetric} if:
	\[
	h(\lambda v) = h(v) \quad \text{for all } \lambda > 0 \text{ and } v \in \overline{\mathbb{R}_{+}^{m+1}},
	\]
	and if there exists a $k$-dimensional subspace $V \subset \mathbb{R}^m = \partial \mathbb{R}_{+}^{m+1}$ such that:
	\[
	h(x + v) = h(x) \quad \text{for all } x \in \overline{\mathbb{R}_{+}^{m+1}},\ v \in V.
	\]
\end{definition}

Now, for any $k \in \{0,1,\cdots,m\}$, we define for a minimizing intrinsic $s$-harmonic map the stratified  sets
\begin{equation}\label{eq: classical stratification}
    \Sigma^{k}(u) = \{x \in \mathcal{O} : \text{no tangent map of } u \text{ is } (k+1)\text{-symmetric at } x\}.
\end{equation}
It is straightforward to verify that
\[
\Sigma^{0}(u) \subset \Sigma^{1}(u) \subset \cdots \subset \Sigma^{m-1}(u) \subset \Sigma^{m}(u) = \mathcal{O}.
\]
Furthermore, $x \notin \Sigma^{m-1}(u)$ means that $u$ has a constant tangent map at $x$. This leads to the following simple observation.

\begin{proposition}\label{rmk: sing}
    For $m \geq 3$, let $s \in (0,1)$, and for $m = 2$, let $s \in (0, 2/3)$. Let $u \in H_{int}^{s}(\mathcal{O}, N)$ be a minimizing intrinsic $s$-harmonic map. Then we have
    \[
    \mathrm{sing}(u) = \Sigma^{m-1}(u) \subset \Sigma,
    \]
    where $\Sigma$ and $\mathrm{sing}(u)$ are defined in \eqref{Sigma} and \eqref{def: singular set}, respectively.
\end{proposition}

\begin{proof}
    We first prove $\mathrm{sing}(u) = \Sigma^{m-1}(u)$.

    Suppose that $x \in \mathcal{O} \setminus \Sigma^{m-1}(u)$; that is, $u$ has a constant tangent map $\varphi$ at $x$. By Corollary \ref{rmk:tangent map}, there exists a subsequence $r_k \to 0$ such that $u_{x,r_k} \to \varphi \equiv C$ in $H_{\mathrm{int}}^{s}(D_r, N)$, $v_{\mathbf{x},r_k} \to \Phi$ in $\dot{W}_a^{1,2}(B_r^+; N)$ for all $r > 0$, where $\mathbf{x} = (x, 0)$, $u_{x,r}(y) = u(x + ry)$, $\left.T\right|_{D_r} \Phi = \varphi$, and $\Phi$ is boundary $0$-homogeneous. Then \eqref{eq:v minimising} and Definition \ref{def:v minimising} imply that $\Phi$ is $E^a$-minimizing relative to $D_r$, hence $\Phi \equiv C$. So the compactness Theorem \ref{thm: compactness} implies that $\Theta_s(v, \mathbf{x}, r) \to 0$ as $r \to 0$, which in turn implies by the $\varepsilon$-regularity theorem that $u$ is smooth in a neighborhood of $x$. Hence $\mathrm{sing}(u) \subset \Sigma^{m-1}(u)$. On the other hand, if $u$ is smooth in a neighborhood of $x$, then surely there is a constant tangent map at $x$. This implies that $\Sigma^{m-1}(u) \subset \mathrm{sing}(u)$.

    By Theorem \ref{thm: partial regularity}, the inclusion $\mathrm{sing}(u) \subset \Sigma$ holds.
\end{proof}

Consequently, in this case, we have
\[
\Sigma^{0}(u) \subset \Sigma^{1}(u) \subset \cdots \subset \Sigma^{m-1}(u) = \mathrm{sing}(u).
\]
This is the so-called classical stratification of $\mathrm{sing}(u)$. In the next section, we will use the approach of Cheeger and Naber \cite{Cheeger-Naber-2013-CPAM} to study each $\Sigma^{k}(u)$.

It is worth noting that in \cite{Roberts-18-CV}, Roberts established an estimate for the Hausdorff dimension of the set $\Sigma$. However, since his work did not investigate the fine structure of the singular set $\mathrm{sing}(u) $, the dimension of $\mathrm{sing}(u)$ could potentially be lower; obtaining an improved estimate for $\dim_{\mathcal{H}} \mathrm{sing}(u)$ is therefore one of the main objectives of this paper.

\section{Quantitative stratification and volume estimates}\label{sec: Quantitative stratification and volume estimates}

Our analysis relies on the technique of quantitative symmetry, originally introduced by Cheeger and Naber \cite{Cheeger-Naber-2013-CPAM}.

\begin{definition}[Quantitative symmetry]
     Let $s \in (0,1)$ for $m \geq 3$, and  let $s \in (0, 2/3)$ for $m = 2$. Fix a constant $1 < \gamma_0 < \min \left\{ \frac{1}{1-s}, 2 \right\}$. Given a minimizing intrinsic $s$-harmonic map $u \in H_{\mathrm{int}}^{s}(\mathcal{O}, N)$, $\varepsilon > 0$, and a nonnegative integer $k$, we say that $u$ is \emph{$(k, \varepsilon)$-symmetric} on $D_r(x) \subset \subset \mathcal{O}$ if there exists a $v \in \dot{W}_a^{1,2}(\mathbb{R}_{+}^{m+1}; N)$ with $\left.T\right|_{\mathcal{O}} v = u$ such that $I^a(u, \mathcal{O}) = E^a(v)$ and a boundary $k$-symmetric function $h: \overline{\mathbb{R}_{+}^{m+1}} \to N$ such that
    \[
    \medint_{B_1^{+}} \left| v_{\mathbf{x},r}(\mathbf{y}) - h_{\mathbf{0},r}(\mathbf{y}) \right|^{\gamma_0}  \mathrm{d}\mathbf{y} = \medint_{B_r^{+}(\mathbf{x})} \left| v(\mathbf{y}) - h(\mathbf{y} - \mathbf{x}) \right|^{\gamma_0}  \mathrm{d}\mathbf{y} \leq \varepsilon,
    \]
    where $\mathbf{x} = (x, 0)$ and $v_{\mathbf{x},r}(\mathbf{y}) = v(\mathbf{x}+r\mathbf{y})$.
\end{definition}

By Lemma~\ref{lemma:Weighted homogeneous Sobolev spaces} and the compact embedding \eqref{eq:compact embedding}, the above integral is well-defined. A basic fact concerning this notion is the following weak compactness property of quantitatively symmetric functions.

\begin{remark}\label{rmk: symmetry compactness}
     Let $s \in (0,1)$ for $m \geq 3$, and  let $s \in (0, 2/3)$ for $m = 2$.  Let $\{u_i\} \subset H_{\Lambda}^{s}(\mathcal{O}, N)$ be a sequence of uniformly bounded minimizing intrinsic $s$-harmonic maps. Suppose $u_i$ converges weakly to a mapping $u$ in $H_{\Lambda}^{s}(\mathcal{O}, N)$, $D_{2r}(x) \subset \mathcal{O}$, and $u_i$ is $(k, \varepsilon_i)$-symmetric on $D_r(x)$ for some $\varepsilon_i \to 0$. Then $u$ is $k$-symmetric on $D_r(x)$.
\end{remark}

To see this, assume without loss of generality that $x = 0$ and $r = 1$. By the definition of quantitative symmetry, there exist a sequence of $v_i \in \dot{W}_a^{1,2}(\mathbb{R}_{+}^{m+1}; N)$ with $\left.T\right|_{\mathcal{O}} v_i = u_i$ such that $I^a(u_i, \mathcal{O}) = E^a(v_i)$ and boundary $k$-symmetric functions $h_i$ such that
\[
\medint_{B_1^{+}} \left| v_i(\mathbf{y}) - h_i(\mathbf{y}) \right|^{\gamma_0}  \mathrm{d}\mathbf{y} \leq \varepsilon_i \to 0.
\]
Since $E^a(v_i) = I^a(u_i, \mathcal{O})\le \Lambda$, we may assume (up to a subsequence) that $v_i \rightharpoonup v$ in $\dot{W}_a^{1,2}(\mathbb{R}_{+}^{m+1}; N)$, $v_i \rightharpoonup v$ in $W_a^{1,2}(B_1^{+}; N)$, and $v_i \to v$ in $L^{\gamma_0}(B_1^{+}; N)$. Then $h_i \to v$ strongly in $L^{\gamma_0}(B_1^{+}; N)$. The fact that $h_i$ is boundary $k$-symmetric implies that $v$ is boundary $k$-symmetric in $B_1^{+}$. On the other hand, by the compactness Theorem~\ref{thm: compactness}, we have $v_i \to v$ in $\dot{W}_a^{1,2}(\mathbb{R}_{+}^{m+1}; N)$. Since $\left.T\right|_{\mathcal{O}}$ is a continuous linear trace operator, it follows that $\left.T\right|_{\mathcal{O}} v = u$, which implies that $u$ is $k$-symmetric in $D_1$.

Given the definition of quantitative symmetry, we can introduce a quantitative stratification for points of a function according to how symmetric it is around those points.

\begin{definition}[Quantitative stratification]\label{def: qs}
    For any minimizing intrinsic $s$-harmonic map $u \in H_{\mathrm{int}}^{s}(\mathcal{O}, N)$, $r, \eta > 0$, and $k \in \{0, 1, \dots, m\}$, we define the \emph{$k$-th quantitative singular stratum} $\mathcal{S}_{\eta,r}^k(u) \subset \mathcal{O}$ by
    \[
    \mathcal{S}^k_{\eta,r}(u) =: \left\{ x \in \mathcal{O} : u \text{ is not } (k+1, \eta)\text{-symmetric on } D_s(x) \text{ for any } r \leq s \leq 1 \right\}.
    \]
    Furthermore, we set
    \[
    \mathcal{S}^k_\eta(u) =: \bigcap_{r > 0} \mathcal{S}^k_{\eta,r}(u)
    \quad \text{and} \quad
    \mathcal{S}^k(u) =: \bigcup_{\eta > 0} \mathcal{S}^k_\eta(u).
    \]
\end{definition}

It is straightforward to verify by definition that if $k' \leq k$, $\eta' \geq \eta$, and $r' \leq r$, then
\[
\mathcal{S}^{k'}_{\eta',r'}(u) \subseteq \mathcal{S}^k_{\eta,r}(u).
\]
The following proposition shows that this definition of quantitative stratification is indeed a quantitative version of the classically defined one.

\begin{proposition}
    For $m \geq 3$, let $s \in (0,1)$, and for $m = 2$, let $s \in (0, 2/3)$. If $u \in H_{\Lambda}^{s}(\mathcal{O}, N)$ is a minimizing intrinsic $s$-harmonic map, then
    \[
    \Sigma^{k}(u) \subset S^{k}(u), \qquad \forall\, 0 \le k \le m,
    \]
    where $\Sigma^{k}(u)$ is defined as in \eqref{eq: classical stratification}.
\end{proposition}

\begin{proof}
    Suppose $x \notin S^{k}(u)$. Then for each $i \geq 1$, there exists $r_{i} > 0$ such that $u$ is $(k+1, 1/i)$-symmetric on $D_{r_{i}}(x)$. That is, there exists a boundary $(k+1)$-symmetric function $h_i: \overline{\mathbb{R}_{+}^{m+1}} \to N$ such that
    \[
    \medint_{B_1^{+}} \left| v_{\mathbf{x},r_{i}}(\mathbf{y}) - h_i(\mathbf{y}) \right|^{\gamma_0}  \mathrm{d}\mathbf{y} \leq 1/i.
    \]
    Up to a subsequence, we may assume that $v_{\mathbf{x}, r_i} \rightharpoonup w$ in $\dot{W}_a^{1,2}(B_1^{+}; N)$, $v_{\mathbf{x}, r_i} \rightharpoonup w$ in $W_a^{1,2}(B_1^{+}; N)$, and $v_{\mathbf{x}, r_i} \to w$ in $L^{\gamma_0}(B_1^{+}; N)$. Then $h_i \to w$ in $L^{\gamma_0}(B_1^{+}; N)$. Moreover, $w$ is boundary $(k+1)$-symmetric in $B_1^{+}$.

    If $r_{i} \to 0$, then by the compactness Theorem~\ref{thm: compactness}, we obtain a tangent map $\left.T\right|_{D_1} w$ of $u$ at $x$ which is $(k+1)$-symmetric on $D_{1}$. If $r_i \to r > 0$, then from Theorem~\ref{thm: compactness},
    \[
    \medint_{B^+_{r_i}(\mathbf{x})} \left| v(\mathbf{y}) - h_i(\mathbf{y} - \mathbf{x}) \right|^{\gamma_0}  \mathrm{d}\mathbf{y} = \medint_{B_1^{+}} \left| v_{\mathbf{x},r_{i}}(\mathbf{y}) - h_i(\mathbf{y}) \right|^{\gamma_0}  \mathrm{d}\mathbf{y} \leq 1/i.
    \]
    Letting $i \to \infty$, the lower semicontinuity implies that
    \[
    \medint_{B^+_r(\mathbf{x})} \left| v(\mathbf{y}) - w(\mathbf{y} - \mathbf{x}) \right|^{\gamma_0}  \mathrm{d}\mathbf{y} = 0.
    \]
    This implies that there exists a positive radius $\delta \leq r$ such that $v(\mathbf{y}) = w(\mathbf{y} - \mathbf{x})$ for almost every $\mathbf{y} \in B^+_\delta(\mathbf{x})$, so $w$ is boundary $(k+1)$-symmetric in $B_1^{+}$. It follows that the tangent map of $v$ at $\mathbf{x}$ is boundary $(k+1)$-symmetric in the whole upper half-space $\mathbb{R}^{m+1}_{+}$. The compactness Theorem~\ref{thm: compactness} and the continuity of the trace operator $T|_{\mathcal{O}}$ imply that $u$ is $(k+1)$-symmetric. Hence, all tangent maps of $u$ at $x$ are $(k+1)$-symmetric, and thus $x \notin \Sigma^k(u)$. Therefore, $\Sigma^{k}(u) \subset S^{k}(u)$.
\end{proof}

	The following two lemmata give a criterion on the quantitative symmetry of a given 	mapping.
	\begin{lemma}[Quantitative Rigidity]\label{quantitative rigidity}
		When $m \geq 3$ let $s \in(0,1)$ and when $m=2$ let $s \in \left(0, 2/3\right)$. Let $u\in H_{\Lambda}^{s}(D_4,N)$  is a minimizing intrinsic $s$-harmonic map, $v \in \dot{W}_a^{1,2}\left(\mathbb{R}_{+}^{m+1} ; N\right)$ with $\left.T\right|_{D_4} v=u$ such that $I^a(u,D_4)=E^a(v)$. Then for every $\ep>0$ and $0<\gamma<1 / 2$ there exist $\delta=\delta(\gamma, \ep, s, m, N, \Lambda)>0$ and $q=q(\gamma, \ep, s, m, N, \Lambda) \in \mathbb{N}$ such that for $r \in(0,1 / 2)$, if
		$$
		\boldsymbol{\Theta}_s(v,\boldsymbol{0},2 r)-\boldsymbol{\Theta}_s\left(v,\boldsymbol{0}, \gamma^q r\right) \leq \delta,
		$$
		then $u$ is $(0, \ep)$-symmetric on $D_{2r}$.	\end{lemma}

	\begin{proof} Assume there exist $\ep>0$ and $0<\gamma<1/2$ for which the statement is false. Again we assume that $r=1$. Then there exist a sequence of minimizing intrinsic $s$-harmonic
		maps  $u_i \in H_{\Lambda}^{s}(D_4,N)$ ($i=1,2,\ldots$) satisfying
		$$
		\boldsymbol{\Theta}_s(v_i,\boldsymbol{0},2 )-\boldsymbol{\Theta}_s\left(v_i,\boldsymbol{0}, \gamma^i \right) \leq \frac{1}{i},
		$$
		but none of $u_i$ is  $(0,\ep)$-symmetric on $D_2$. Up to a subsequence, we may assume that $v_i\wto v $ in $\dot{W}_a^{1,2}\left(\mathbb{R}_{+}^{m+1} ; N\right)$, $v_i\wto v $ in $W_a^{1,2}\left(B_{2}^+ ; N\right)$, $v_i\to v $ in $L^{\gamma_0}\left(B_{2}^+; N\right)$. Then the Monotonicity formula \eqref{monotonicity formula} implies that
		$$
		\int_{B_{2}^+}z^a|\mathbf{x}\cdot \nabla v|^2d\mathbf{x}\leq C\liminf_{i\rightarrow \infty} \int_{B_2^+}z^a\frac{|\mathbf{x}\cdot \nabla v_i|^2}{|\mathbf{x}|^{m+a+1}}d\mathbf{x}=0.
		$$			
		Thus, $v$ is boundary $0$-homogeneous in $B_{2}^+$. Extending $v$ to the whole upper space $\mathbb{R}^{m+1}_{+}$ by homogeneity, we denote it also by $v$. But then,  the strong convergence of $v_i\to v$ in $L^{\gamma_0}\left(B_2^+ ; N\right)$ implies that
		$	\medint_{B_2^+}\left|v_i-v\right|^{\gamma_0}<\ep	$
		for $i$ sufficiently large.
		Hence, $u_i$ is $(0,\ep)$-symmetric on $D_2$, which gives a contradiction.
	\end{proof}
	In the above proof we need to assume $u$ is a minimizing intrinsic $s$-harmonic map so as to use the monotonicity formula  \eqref{monotonicity formula}. The following lemma gives a quantitative geometric description on $k$-symmetry.
	
	\begin{lemma}[Quantitative cone splitting]\label{quantitative cone splitting}
		 When $m \geq 3$ let $s \in(0,1)$ and when $m=2$ let $s \in\left(0, 2/3\right)$. Given constants $\eta, \tau,\Lambda>0$, there exists $\ep=\ep(s, m, N, \Lambda$, $\eta, \tau)>0$ such that the following holds. Let $u\in H_{\Lambda}^{s}(D_4,N)$  is a minimizing intrinsic $s$-harmonic map, $x \in D_1$ and $0<r<1$. If
 $x \in \mathcal{S}_{\eta, r}^k(u)$ and $u$ is $(0, \ep)$-symmetric on $D_{2r}(x)$,		then there exists a $k$-dimensional affine subspace $V \subset  \mathbb{R}^{m}$ such that
		$$
		\left\{y \in D_r(x): u \text { is }(0,\ep)\text{-symmetric on } D_{2r}(y)\right\} \subset T_{\tau r}(V) \text {. }
		$$	
	\end{lemma}
	\begin{proof}	Assume without loss of generality that $x=0$ and $r=1$. We use a contradiction argument. Thus, for  given $\eta, \tau>0$, there exist a sequence $\left\{u_i\right\}$ with $I^{a}(u_i,D_4)\le\La$ such that $0 \in \mathcal{S}_{\eta, 1}^k\left(u_i\right)$ for all $i, u_i$ is $(0,1/i)$-symmetric on $D_2$ and there exist points $\left\{x_1^i, x_2^i, \ldots, x_{k+1}^i\right\} \subset D_1$ satisfying the following two conditions:
		
		$\bullet$ $u_i$ is $(0,1/i)$-symmetric on each $D_2(x_j^i)$ for $j=1, \ldots, k+1$. That is, there exists a $v_i \in \dot{W}_a^{1,2}\left(\mathbb{R}_{+}^{m+1} ; N\right)$ with $\left.T\right|_{D_4} v_i=u_i$ such that $I^a(u,D_4)=E^a(v)$,
		\[\int_{B^+_{2}(\mathbf{x}_j^i)} |v_i-h_{ij}|^{\gamma_0} d\mathbf{x}\le 1/i\] for some boundary $0$-symmetric function $h_{ij}$.
		
		$\bullet$   dist$\left(x_j^i, \operatorname{span}\left\{0, x_1^i, \ldots, x_{j-1}^i\right\}\right) \geq \tau$ for all $j=1, \ldots, k+1$.
		
		After passing to a subsequence, there exists a map $v$ such that         $v_i \wto v$ in $\dot{W}_a^{1,2}\left(\mathbb{R}_{+}^{m+1} ; N\right)$, $v_i\wto v $ in $W_a^{1,2}\left(B_{4}^+ ; N\right)$, $v_i\to v $ in $L^{\gamma_0}\left(B_{4}^+; N\right)$; and there exist points $\left\{x_1, \ldots, x_{k+1}\right\} \subset \overline{D}_1$ such that v is  boundary $0$-symmetric on  $B_2^{+}\left(x_j\right)$ for all $ j=0,1, \ldots, k+1$.
The distance relations are also preserved:  we have  dist$\left(x_j, \operatorname{span}\left\{x_0, x_1, \ldots, x_{j-1}\right\}\right) \geq \tau$ for all $j=0, \ldots, k+1$.
		
		It is now straightforward to verify that
$v$ is boundary $(k+1)$-symmetric on $B_1^{+}$.  But then,  the strong convergence  $v_i\to v $ in $L^{\gamma_0}\left(B_{1}^+; N\right)$  gives a contradiction to the assumption $0 \in \mathcal{S}_{\eta, 1}^k\left(u_i\right)$ for $i\gg 1$.		\end{proof}

	To continue, let us introduce the following notation.
	
	\begin{definition} Let $u\in H_{\Lambda}^{s}(D_4,N)$  is a minimizing intrinsic $s$-harmonic map, $v \in \dot{W}_a^{1,2}\left(\mathbb{R}_{+}^{m+1} ; N\right)$ with $\left.T\right|_{D_4} v=u$ such that $I^a(u,D_4)=E^a(v)$. For $x \in D_1$ and $0 \leq s<t<1$, denote
		$$
		\mathcal{W}_{s, t}(x, u):=\boldsymbol{\Theta}_s(v,\mathbf{x}, t)-\boldsymbol{\Theta}_s(v,\mathbf{x},s) \geq 0 .
		$$
	\end{definition}
	
	Notice that, by monotonicity formula \eqref{monotonicity formula}, for $\left(s_1, t_1\right),\left(s_2, t_2\right)$ with $t_1 \leq s_2$,
	$$
	\mathcal{W}_{s_1, t_1}(x, u)+\mathcal{W}_{s_2, t_2}(x, u) \leq \mathcal{W}_{s_1, t_2}(x, u)
	$$
	with equality if $t_1=s_2$.
	Given constants $0<\gamma<1 / 2$ , $\delta>0$ and $q \in \mathbb{Z}^{+}$(these parameters will be fixed suitably in Lemma \ref{covering lemma}),  let $Q$ be the number of positive integers $j$ such that
	$$
	\mathcal{W}_{\gamma^{j+q}, \gamma^j}(x, u)>\delta.
	$$
	Then there has
	\begin{equation}\label{Q}
		Q \leq \frac{C_1(q+1)}{4^{m-2s}}\Lambda \delta^{-1},
	\end{equation}
	where $C_1=C_1(m,s)$. To see this, just note that
	$$
	Q \de \le \sum_{j=1}^\wq \mathcal{W}_{\gamma^{j+q}, \gamma^j}(x, u) \leq (q+1)\mathcal{W}_{0,1}(x, u) \leq (q+1) \boldsymbol{\Theta}_s(v,\boldsymbol{0},2)\leq \frac{C_1(q+1)}{4^{m-2s}}E^{a}\left(v\right).
	$$

	Following \cite{Cheeger-H-N-2013, Cheeger-H-N-2015, Cheeger-Naber-2013-Invent, Cheeger-Naber-2013-CPAM}, for each $x \in D_3$, we define a sequence $\left\{T_j(x)\right\}_{j \geq 1}$ with values in $\{0,1\}$ as the following manner. For each $j \in \mathbb{Z}^{+}$define
	$$
	T_j(x)= \begin{cases}
		1, \quad \text { if } \mathcal{W}_{\gamma^{j+q}, \gamma^j}(x, u)>\delta, \\
		0, \quad\text { if } \mathcal{W}_{\gamma^{j+q}, \gamma^j}(x, u)\leq\delta.
	\end{cases}
	$$
	\eqref{Q} implies that
	$$\sum_{j\ge 1}T_j(x)\le Q,\qquad \forall\,x\in D_3.$$
	That is,  there exist at most $Q$ nonzero entries in the sequence.  Thus, for each $\alpha$-tuple $T^\alpha=\left(T_j^\alpha\right)_{1 \leq j \leq \alpha}$ with entries in $\{0,1\}$, by defining
	$$
	E_{T^\alpha}(u)=\left\{x \in D_1 \mid T_j(x)=T_j^\alpha \text { for } 1 \leq j \leq \alpha\right\},
	$$
	we obtain a  decomposition of $D_1$ by at most $\binom{\alpha}{Q} \leq \alpha^Q$ non-empty such sets $E_{T^\alpha}(u)$, even though \textit{a priori} there are  $2^\alpha$ choices of such $\alpha$-tuple. This estimate plays an important role in the volume estimates below.
	
	\begin{lemma}[Covering Lemma]\label{covering lemma} When $m \geq 3$ let $s \in(0,1)$ and when $m=2$ let $s \in \left(0, 2/3\right)$. There exists $c_0(m)<\infty$ such that, for each $\alpha\ge 1$, the set $\mathcal{S}_{\eta, \gamma^\alpha}^j(u) \cap E_{T^\alpha}(u)$ can be covered by at most $c_0\left(c_0 \gamma^{-m}\right)^Q\left(c_0 \gamma^{-j}\right)^{\alpha-Q}$ balls of radius $\gamma^\alpha$.
	\end{lemma}
	\begin{proof}
		For fixed $\eta, \gamma$, let $\tau=\gamma$ and choose $\ep$ as in Lemma \ref{quantitative cone splitting}. For this $\ep, \gamma$, by Lemma \ref{quantitative rigidity} there exist $\delta>0$ and $q \in \mathbb{Z}^{+}$such that if
		$$
		\boldsymbol{\Theta}_s(v,\mathbf{x},2\gamma^j)-\boldsymbol{\Theta}_s(v,\mathbf{x},\gamma^{j+q})\leq \delta,
		$$
		then $u$ is $\left(0,\ep\right)$-symmetric on $D_{2\gamma^j}(x)$. Fix this $\delta, q$ throughout the proof and define $T_j(x)$ accordingly.
		
		We now determine the covering by induction argument. For $\alpha=0$, we can simply choose a minimal covering of $\mathcal{S}_{\eta, 1}^j(u) \cap D_1$ by at most $c(m)$ balls of radius $1$ with centers in $\mathcal{S}_{\eta, 1}^j(u) \cap D_1$. Suppose now the statement holds for all $\alpha$-tuples, and given a $\alpha+1$ tuple $T^{\alpha+1}$. Here are two simple observations. First, by definition, we have $\mathcal{S}_{\eta, \gamma^{\alpha+1}}^j(u) \subset \mathcal{S}_{\eta, \gamma^\alpha}^j(u)$. Next, by denoting $T^\alpha$ the $\alpha$ tuple obtained by dropping the last entry from $T^{\alpha+1}$, we immediately get $E_{T^{\alpha+1}}(u) \subset E_{T^\alpha}(u)$.
		
		We determine the covering recursively. For each ball $D_{\gamma^\alpha}(x)$ in the covering of $\mathcal{S}_{\eta, \gamma^\alpha}^j(u)\cap E_{T^\alpha}(u)$, we will take a minimal covering of $D_{\gamma^\alpha}(x) \cap \mathcal{S}_{\eta, \gamma^\alpha}^j(u) \cap E_{T^\alpha}(u)$ by balls of radius $\gamma^{\alpha+1}$ as follows.  In the case $T_\alpha^\alpha=1$, employing a standard volume argument, we can bound on the number of balls to get an upper bound on the covering by
		$$
		c(m) \gamma^{-m}.
		$$
		
	In the other case $T_\alpha^\alpha=0$, we can do better. In this case we have both $T_\alpha(x)=0$ and $T_\alpha(y)=0$ for all $y\in D_{\gamma^\alpha}(x)\cap E_{T^\alpha}(u)$, hence $\mathcal{W}_{\gamma^{\alpha+q}, \gamma^\alpha}(x, u)\leq \delta$, $\mathcal{W}_{\gamma^{\alpha+q}, \gamma^\alpha}(y, u)\leq \delta$.  By the choice of $\delta, q$, this implies that $u$ is $\left(0,\ep\right)$-symmetric on $D_{2\gamma^\alpha}(x)$ and $u$ is $\left(0,\ep\right)$-symmetric on $D_{2\gamma^\alpha}(y)$.  Recall that  $x \in \mathcal{S}_{\eta, \gamma^\alpha}^j(u)$. Hence we can apply Lemma \ref{quantitative cone splitting} to conclude that the set $E_{T^\alpha}(u) \cap D_{\gamma^\alpha}(x)$ is contained in a $\gamma^{\alpha+1}$ tubular neighborhood of some $j$ dimensional plane $V$.  Therefore, in this case, we can cover the intersection with the smaller number of balls
		$$
		c(m) \gamma^{-j}.
		$$
		
		Given any $\alpha>0$ and $E_{T^\alpha}(u)$, the number of times we need to apply the weaker estimate is bounded above by $Q$. Thus, the proof is complete.
	\end{proof}

	We are now ready to prove Theorem \ref{volume estimate}.
	
	\begin{proof}[Proof of Theorem \ref{volume estimate}] Choose $\gamma<1 / 2$ such that $\gamma \leq c_0^{-2 / \eta}$, where $c_0$ is as in Lemma \ref{covering lemma}. Then $c_0^\alpha \leq(\gamma^\alpha)^{-\eta / 2}$ and since exponentials grow faster than polynomials,
		$$
		\alpha^Q \leq c(Q) c_0^\alpha \leq c(\eta, m, Q)\left(\gamma^\alpha\right)^{-\eta / 2} .
		$$
		
		Since	$ \operatorname{Vol}\left(D_{\gamma^\alpha}(x)\right)=\omega_m \gamma^{\alpha m}$
		and $D_1$ can be decomposed into at most $\alpha^Q$ sets $E_{T^\alpha}(u)$ for any $\alpha$, we have
		$$
		\begin{aligned}
			\operatorname{Vol}\left(T_{\gamma^\alpha}\left(\mathcal{S}_{\eta, \gamma^\alpha}^j\right) \cap D_1\right) & \leq \alpha^{Q}\left[\left(c_0 \gamma^{-m}\right)^Q\left(c_0 \gamma^{-j}\right)^{\alpha-Q}\right] \omega_m \gamma^{\alpha m} \\
			& \leq c(m, Q, \eta) \alpha^Q c_0^\alpha\left(\gamma^\alpha\right)^{m-j} \\
			& \leq c(m, Q, \eta)\left(\gamma^\alpha\right)^{m-j-\eta}.
		\end{aligned}
		$$
		Thus, for any $0<r<1$, by choosing $\alpha>0$ such that $\gamma^{\alpha+1} \leq r<\gamma^\alpha$, we deduce that
		$$
		\begin{aligned}
			\operatorname{Vol}\left(T_r\left(\mathcal{S}_{\eta, r}^j\right) \cap D_1\right) & \leq \operatorname{Vol}\left(T_{\gamma^\alpha}\left(\mathcal{S}_{\eta, \gamma^\alpha}^j\right) \cap D_1\right) \\
			& \leq c(m, Q, \eta)\left(\gamma^\alpha\right)^{m-j-\eta} \\
			& \leq c(m, Q, \eta)\left(\gamma^{-1} r\right)^{m-j-\eta} \\
			& \leq c(m, s, \eta, N, \Lambda) r^{m-j-\eta}.
		\end{aligned}
		$$
		The proof is complete.
	\end{proof}

	\section{Proof of Theorems \ref{thm: integrability} and \ref{thm: regularity scale estimate}}\label{sec: regularity results}

We first prove another type of $\varepsilon$-regularity theorem.

\begin{theorem}\label{thm: sym-to-reg}
    For $m \geq 3$, let $s \in [1/2, 1)$, and for $m = 2$, let $s \in [1/2, 2/3)$. Given $\Lambda > 0$, assume that $u \in H_{\Lambda}^{s}(D_4, N)$ is a minimizing intrinsic $s$-harmonic map. There exists a constant $\delta = \delta(m, \Lambda, s, N) > 0$ such that, if $u$ is $(m-1, \delta)$-symmetric on $D_{2}$, then
    \[
    r_{u}(0) \geq \theta,
    \]
    where $\theta = \theta(m, N, s) \in (0,1)$ is the constant defined in Theorem~\ref{thm: partial regularity}.
\end{theorem}
	
	The proof relies on the following lemmata.	The first
	is an $\ep$-regularity lemma, which shows that high order
	symmetry implies regularity.
	
	\begin{lemma}[$(m,\epsilon)$-Regularity]\label{lemma: new epsilon regularity}
		When $m \geq 3$ let $s \in(0,1)$ and when $m=2$ let $s \in \left(0, 2/3\right)$. Given $\Lambda>0$ and assume $u\in H_{\Lambda}^{s}(D_4,N)$ is a minimizing intrinsic $s$-harmonic map.  There exists $\ep(m,\La,s,N)>0$ such that
		\[
		r_{u}(0)\ge\theta
		\]
		whenever $u$ is $(m,\ep)$-symmetric on $D_2$.
		\end{lemma}
	\begin{proof}
		Suppose, on the contrary, that there exist a sequence of minimizing intrinsic $s$-harmonic
		maps $u_{i}\in H_{\Lambda}^{s}(D_4,N)$ such that $u_{i}$ is $(m,1/i)$-symmetric on $D_2$ and  $r_{u_{i}}(0)<\theta$. By Theorem \ref{thm: compactness},
		we can assume that there exists a map $v$ such that         $v_i \to v$ in $\dot{W}_a^{1,2}\left(B_{2}^+ ; N\right)$, and $v_i\to v $ in $L^{\gamma_0}\left(B_{2}^+; N\right)$,  where $v_i \in \dot{W}_a^{1,2}\left(\mathbb{R}_{+}^{m+1} ; N\right)$ with $\left.T\right|_{D_4} v_i=u_i$ such that $I^a(u_i,D_4)=E^a(v_i)$.
		Letting $i\to \wq$ we find that $v$ is 0-homogeneous and translation invariant with respect to the subspace $\R^m\subset \R^{m+1}$ in $B_{2}^+$. This implies that $v\equiv \text{const}$ in $B_{2}^+$. But then, by the strong convergence we know that
		$$
		\boldsymbol{\Theta}_s(v_{i},\boldsymbol{0},2)\rightarrow 0\quad \text{as}\quad i\rightarrow \infty,
		$$
		which implies that $r_{u_{i}}(0)\ge\theta$  for $i \gg 1$ by Theorem \ref{thm: partial regularity}.		We reach a contradiction.
	\end{proof}

The second ingredient of the proof of Theorem \ref{thm: sym-to-reg} is the following symmetry self-improvement
	lemma.

	\begin{lemma}[Symmetry self-improvement]\label{lem: Symmetry self-improvement}
		When $m \geq 3$ let $s \in[1/2,1)$ and when $m=2$ let $s \in \left[1/2, 2/3\right)$. Given $\Lambda>0$ and assume $u\in H_{\Lambda}^{s}(D_4,N)$ is a minimizing intrinsic $s$-harmonic map. Then, for any $\ep>0$, there
		exists $\de>0$ such that if $u$ is $(m-1,\de)$-symmetric on $D_{2}$, then $u$ is also $(m,\ep)$-symmetric
		on $D_{2}$. \end{lemma}
	\begin{proof}
We proceed by contradiction. Note that $a=1-2s\leq0$. Suppose, for some $\ep_{0}>0$ and for each $i\ge1$, there is a sequence of minimizing intrinsic $s$-harmonic maps $u_{i}\in H_{\Lambda}^{s}(D_4,N)$ such that $u_{i}$ is  $(m-1,1/i)$-symmetric but not $(m,\ep_{0})$-symmetric on $D_2$. By Theorem \ref{thm: compactness}, we can assume that there exists a map $v$ such that $v_i \to v$ in $\dot{W}_a^{1,2}\left(B_{2}^+ ; N\right)$, %$v_i\wto v $ in $W_a^{1,2}\left(B_{2}^+ ; N\right)$,
$v_i\to v $ in $L^{\gamma_0}\left(B_{2}^+; N\right)$,  where $v_i \in \dot{W}_a^{1,2}\left(\mathbb{R}_{+}^{m+1} ; N\right)$ with $\left.T\right|_{D_4} v_i=u_i$ such that $I^a(u_i,D_4)=E^a(v_i)$. Then $u=\left.T\right|_{D_4}v$ is $(m-1)$-symmetric but not $\left(m, \epsilon_0\right)$-symmetric on $D_2$. We first establish that
$$
\int_{B_2^{+}}z^{a}\left|\nabla v\right|^2=0.
$$
Suppose not. Then by radial invariance, for some $c>0$ and $\boldsymbol{x}=(x, z) \in \overline{\mathbb{R}_{+}^{m+1}} \cap \partial B_1$, we have
\begin{equation}\label{eq 4.1}
\int_{B_{1 / 2}(\boldsymbol{x}) \cap B_2^{+}}z^{a}\left|\nabla v\right|^2=c.
\end{equation}
By radial invariance, we obtain
\begin{equation}\label{eq 4.2}
\int_{B_{1 / (2\cdot4^{j})}\left(\boldsymbol{x} / 4^j\right) \cap B_2^{+}}z^{a}\left|\nabla v\right|^2=c\left(4^{-j}\right)^{m-1+a},
\end{equation}
for all $j \in \mathbb{N}$. Notice that the sets $B_{4^{-j} / 2}\left(\boldsymbol{x} / 4^j\right) \cap B_2^{+}$are mutually disjoint.		
Since $v$ is boundary $(m-1)$-symmetric on $B_2^{+}$, there exists $c_1>0$ such that for each $j \in \mathbb{N}$ there exists a collection of at least $c_1\left(4^j\right)^{m-1}$ mutually disjoint sets in $ B_2^{+}$, each of radius $ 1/ 2\cdot 4^{j}$ and such that on each of them, the equality \eqref{eq 4.2} holds. Moreover, when $j_1 \neq j_2$, the collections are obviously disjoint. Therefore,
$$
\int_{B_2^{+}}z^{a}\left|\nabla v\right|^2 \geq \sum_{j=1}^{\infty} c\left(4^{-j}\right)^{m-1+a} c_1\left(4^j\right)^{m-1}=\infty.
$$
But this contradicts the $E^a(v_i)<\Lambda$ for each $i\geq1$ and the strong convergence $v_i \rightarrow v$ in $\dot{W}_a^{1,2}\left(B_2^{+} ; N\right)$. Therefore, $\int_{B_2^{+}}z^{a}\left|\nabla v\right|^2=0$ and then $\int_{B_2^{+}}\left|\nabla v\right|^2=0$. Now, by the Poincar\'{e} inequality,

$$
\int_{B_2^{+}}\left|v-\left(v\right)_2\right|^2 \leq C \int_{B_2^{+}}\left|\nabla v\right|^2=0,
$$
where $\left(v\right)_2$ denotes the average of $v$ over $B_2^{+}$. Thus $v$ is boundary $m$-symmetric on $B_2^{+}$ and $u$ is $m$-symmetric in $D_2$, which gives a contradiction.

 % Case $s\in (1/2,1)$, i.e., $a\in (-1, 0)$. Since $u$ is $(m-1)$-symmetric but not $\left(m, \epsilon_0\right)$ symmetric on $D_2$. Now, we infer that $u$ only depends on the $x_1$-variable, that is $u(x)=\tilde{u}(x_1)$. Then by the homogeneity, we have
	%	\[ \tilde{u}(x_1)=\begin{cases}a,& \text{ if } x_1>0\\
	%	b,& \text{ if } x_1<0
	%	\end{cases}\]
	%	for some $a,b\in N$. It is known that  $a\neq b$ implies $[\tilde{u}]_{H^{1/2}((-1,1))}=+\wq$, so $||\tilde{u}||_{H^{1/2}((-1,1))}=+\wq$. On the other hand, by Lemma \ref{lemma:function space}, Lemma \ref{lemma:Weighted homogeneous Sobolev spaces} and Grisvard \cite{Grisvard-1985}, we obtain that $\dot{W}_a^{1,2}\left(\mathbb{R}_{+}^{m+1} ; N\right)\hookrightarrow H^{1/2}(D_2,N)$, which gives a contradiction. Hence $a=b$, which means that $\tilde{u}$ is a constant map and so $m$-symmetric in $D_2$, again a contradiction.
	%		The proof is complete.
	\end{proof}

	Now we can prove Theorem \ref{thm: sym-to-reg}.

	\begin{proof}[Proof of Theorem \ref{thm: sym-to-reg}] It follows from Lemma  \ref{lemma: new epsilon regularity} and  Lemma \ref{lem: Symmetry self-improvement} directly. \end{proof}
	
	Now we are ready to prove Theorems \ref{thm: integrability} and \ref{thm: regularity scale estimate}.
	
	\begin{proof}[Proof of Theorem \ref{thm: regularity scale estimate}]
We first consider the case $s \in[1/2,1), m\geqslant 3$ or $s \in[1/2,2/3), m=2$. If $x\in \mathcal{B}_{ r}(u)$, then by Theorem \ref{thm: sym-to-reg}, $u$ is not $(m-1,\delta)$-symmetric on $D_{2r/\theta}^+(x)$. In other words, $x\in \mathcal{S}_{\eta, 2r/\theta}^{m-2}$ for any $0<\eta\leq \de(m,\La,s,N)$, the constant defined as in  Theorem \ref{thm: sym-to-reg}. Therefore, we have
		\[
		\mathcal{B}_r(u)\subset \mathcal{S}_{\eta, 2r/\theta}^{m-2}, \qquad \forall\, 0<\eta\leq \de(m,\La,s,N).
		\]
		Then Theorem \ref{volume estimate} yields
		\[
		\operatorname{Vol}\left(T_r\left(\mathcal{B}_r(u)\right) \cap D_1(x)\right) \leq \operatorname{Vol}\left(T_r\left(\mathcal{S}_{\eta, 2 r/\theta}^{m-2}\right) \cap D_1(x)\right) \leq C(m,s,  \Lambda,N, \eta) r^{2-\eta}.
		\]		
For the case $s \in(0,1/2)$, the proof is similar. By Lemma \ref{lemma: new epsilon regularity}, we have
\[
		\mathcal{B}_r(u)\subset \mathcal{S}_{\eta, 2r/\theta}^{m-1}, \qquad \forall\, 0<\eta\leq \delta(m,\La,s,N).
		\]
Then
\[
		\operatorname{Vol}\left(T_r\left(\mathcal{B}_r(u)\right) \cap D_1(x)\right) \leq \operatorname{Vol}\left(T_r\left(\mathcal{S}_{\eta, 2 r/\theta}^{m-1}\right) \cap D_1(x)\right) \leq C(m,s,  \Lambda, N,\eta) r^{1-\eta}.
		\]	
		The proof is complete.
	\end{proof}

	\begin{proof}[Proof of Theorem \ref{thm: integrability}]  This theorem follows from the simple observation that
		\[\{x\in D_1: |\na u(x)|>1/r \} \subset \{x\in D_1: r_u(x)<r \}  \]
		and the volume estimate of  Theorem \ref{thm: regularity scale estimate}. 		
	\end{proof}
	
\vskip 0.3cm
%{\bf Acknowledgement.} The authors would like to thank the anonymous referee for his/her very useful comments that greatly improved our exposition.
%\medskip

\textbf{Funding} This work is partially supported by the NSFC (Nos. 12271195, 12271296), the NSF of Hubei Province (No. 2024AFA061) and
the Open Research Fund of Hubei Key Laboratory of Mathematical Sciences (No.MPL2025ORG018), Central China	Normal University, P. R. China.

\medskip

\textbf{Data Availability Statement} This manuscript has no associated data.
\medskip

\textbf{\Large Declarations}
\medskip

{\bf Conflict of Interest}\quad The authors declare no
conflict of interest.

\end{document}